\documentclass[11pt,a4paper]{amsart}
\usepackage{amssymb,amsmath,epsfig,graphics,mathrsfs}

\usepackage{fancyhdr}
\usepackage{hyperref}
\hypersetup{
 colorlinks   = true,
 urlcolor     = blue,
 linkcolor    = blue,
 citecolor   = red ,
 bookmarksopen=true
}

\newtheorem{thrm}{Theorem}[section]
\newtheorem{lemma}[thrm]{Lemma}
\newtheorem{prop}[thrm]{Proposition}
\newtheorem{cor}[thrm]{Corollary}
\newtheorem{rmrk}[thrm]{Remark}
\newtheorem{dfn}[thrm]{Definition}

\begin{document}

\newcommand{\SL}{\mathcal L^{1,p}( D)}
\newcommand{\Lp}{L^p( Dega)}
\newcommand{\CO}{C^\infty_0( \Omega)}
\newcommand{\Rn}{\mathbb R^n}
\newcommand{\Rm}{\mathbb R^m}
\newcommand{\R}{\mathbb R}
\newcommand{\Om}{\Omega}
\newcommand{\Hn}{\mathbb H^n}
\newcommand{\aB}{\alpha B}
\newcommand{\eps}{\ve}
\newcommand{\BVX}{BV_X(\Omega)}
\newcommand{\p}{\partial}
\newcommand{\IO}{\int_\Omega}
\newcommand{\bG}{\boldsymbol{G}}
\newcommand{\bg}{\mathfrak g}
\newcommand{\bz}{\mathfrak z}
\newcommand{\bv}{\mathfrak v}
\newcommand{\Bux}{\mbox{Box}}
\newcommand{\e}{\ve}
\newcommand{\X}{\mathcal X}
\newcommand{\Y}{\mathcal Y}
\newcommand{\W}{\mathcal W}
\newcommand{\la}{\lambda}
\newcommand{\La}{\Lambda}
\newcommand{\vf}{\varphi}
\newcommand{\rhh}{|\nabla_H \rho|}
\newcommand{\Ba}{\mathscr{B}_a}
\newcommand{\Za}{Z_\beta}
\newcommand{\ra}{\rho_\beta}
\newcommand{\na}{\nabla_\beta}
\newcommand{\vt}{\vartheta}
\newcommand{\G}{\Gamma}
\newcommand{\Ga}{\mathbb G}
\newcommand{\HHa}{\mathscr H_a}
\newcommand{\HH}{\mathscr H}
\newcommand{\paa}{z^a \p_z}

\numberwithin{equation}{section}

\newcommand{\RN} {\mathbb{R}^N}
\newcommand{\Sob}{S^{1,p}(\Omega)}
\newcommand{\Dxk}{\frac{\partial}{\partial x_k}}
\newcommand{\Co}{C^\infty_0(\Omega)}
\newcommand{\Je}{J_\ve}
\newcommand{\beq}{\begin{equation}}
\newcommand{\bea}[1]{\begin{array}{#1} }
\newcommand{\eeq}{ \end{equation}}
\newcommand{\ea}{ \end{array}}
\newcommand{\eh}{\ve h}
\newcommand{\Dxi}{\frac{\partial}{\partial x_{i}}}
\newcommand{\Dyi}{\frac{\partial}{\partial y_{i}}}
\newcommand{\Dt}{\frac{\partial}{\partial t}}
\newcommand{\aBa}{(\alpha+1)B}
\newcommand{\GF}{\psi^{1+\frac{1}{2\alpha}}}
\newcommand{\GS}{\psi^{\frac12}}
\newcommand{\HFF}{\frac{\psi}{\rho}}
\newcommand{\HSS}{\frac{\psi}{\rho}}
\newcommand{\HFS}{\rho\psi^{\frac12-\frac{1}{2\alpha}}}
\newcommand{\HSF}{\frac{\psi^{\frac32+\frac{1}{2\alpha}}}{\rho}}
\newcommand{\AF}{\rho}
\newcommand{\AR}{\rho{\psi}^{\frac{1}{2}+\frac{1}{2\alpha}}}
\newcommand{\PF}{\alpha\frac{\psi}{|x|}}
\newcommand{\PS}{\alpha\frac{\psi}{\rho}}
\newcommand{\ds}{\displaystyle}
\newcommand{\Zt}{{\mathcal Z}^{t}}
\newcommand{\XPSI}{2\alpha\psi \begin{pmatrix} \frac{x}{|x|^2}\\ 0 \end{pmatrix} - 2\alpha\frac{{\psi}^2}{\rho^2}\begin{pmatrix} x \\ (\alpha +1)|x|^{-\alpha}y \end{pmatrix}}
\newcommand{\Z}{ \begin{pmatrix} x \\ (\alpha + 1)|x|^{-\alpha}y \end{pmatrix} }
\newcommand{\ZZ}{ \begin{pmatrix} xx^{t} & (\alpha + 1)|x|^{-\alpha}x y^{t}\\
     (\alpha + 1)|x|^{-\alpha}x^{t} y &   (\alpha + 1)^2  |x|^{-2\alpha}yy^{t}\end{pmatrix}}
\newcommand{\norm}[1]{\lVert#1 \rVert}
\newcommand{\ve}{\varepsilon}
\newcommand{\Rnn}{\mathbb R^{n+1}}
\newcommand{\Rnp}{\mathbb R^{n+1}_+}
\newcommand{\B}{\mathbb{B}}
\newcommand{\Ha}{\mathbb{H}}
\newcommand{\xx}{\mathscr X}
\newcommand{\Sa}{\mathbb{S}}
\newcommand{\x}{\nabla_\mathscr X}
\newcommand{\LL}{\mathscr L}
\newcommand{\Ls}{(-\LL)^s}

\title[Some properties of sub-Laplaceans]
{Some properties of sub-Laplaceans}

\author{Nicola Garofalo}

\address{Dipartimento d'Ingegneria Civile e Ambientale (DICEA)\\ Universit\`a di Padova\\ Via Marzolo, 9 - 35131 Padova,  Italy}
\vskip 0.2in
\email{nicola.garofalo@unipd.it}

\thanks{The author was supported in part by a Progetto SID (Investimento Strategico di Dipartimento) ``Non-local operators in geometry and in free boundary problems, and their connection with the applied sciences", University of Padova, 2017.}

\dedicatory{Ad Anna Aloe, amica dolce e indimenticabile}

\maketitle

\tableofcontents

\section{Preamble}\label{S:pre}

This note is dedicated to the memory of my dearest friend Anna Salsa, n\'ee Aloe. I cannot speak of my deep connection with Anna without associating it to my friendship with Sandro, her husband and companion of more than forty years. I met Anna and Sandro for the first time in 1979 in Cortona, when I was a $o(1)$. I was sitting on the stairs outside the Oasi Neumann, idly playing my guitar, wasting time after a day of lectures. Anna and Sandro had just arrived in Cortona to visit Gene Fabes, one of the two lecturers of the summer school I was attending. They had recently returned from Minneapolis, where they had spent one milestone year. Gene had invited Sandro to work with him, and Anna went along taking a leave of absence from her job as a teacher. We became friends the moment we met...that late afternoon of almost forty years ago. Through the years our connection has increasingly deepened. Anna and Sandro became one of the key presences in my life, a certainty on which I could always lean on. Spending time at their home in Novara was literally like going home. Through the years Anna has been an incredibly unique friend. She had charm, intelligence, sense of humor and an exclusive way of connecting to people around her. I miss her deeply.  

\section{Introduction}\label{S:intro}

In this note I present some properties of sub-Laplaceans associated with a collection of $C^\infty$ vector fields $\mathscr X = \{X_1,...,X_m\}$ satisfying H\"ormander's finite rank assumption in $\Rn$. As it is well-known from the fundamental work \cite{H}, the sub-Laplacean associated with such system,
\begin{equation}\label{LL}
\mathscr L = - \sum_{i=1}^m X^\star_iX_i,
\end{equation}
is a second order hypoelliptic operator. Here, we have denoted by $X_i^\star$ the formal adjoint of the vector field $X_i$. The operator $-\LL$ is positive and in divergence form, and it admits a positive fundamental solution $\Gamma(x,y)$ which is $C^\infty$ outside the diagonal. We note explicitly that $\mathscr L$ is formally self-adjoint, and thus $\Gamma(x,y) = \Gamma(y,x)$. As it will clearly appear the three aspects that primarily enter into our considerations, following an approach that was proposed in the author's lecture notes of a 1991 summer school in Cortona \cite{Gcortona} are:
\begin{itemize}
\item divergence structure of $\mathscr L$;
\item hypoellipticity;
\item existence of a (smooth) strictly positive fundamental solution.
\end{itemize}
These three aspects have been extensively used in our previous joint works \cite{CGL}, \cite{CDGtorino}, \cite{CDGmz}, \cite{DG}. 

This note is organized as follows: in Section \ref{S:nsw} I recall a fundamental result of Nagel, Stein and Wainger in \cite{NSW} concerning the local size of the metric balls with respect to the distance naturally associated with  \eqref{LL}. In Section \ref{S:fs} I recall the size estimates of the fundamental solution independently established by Sanchez-Calle \cite{SC} and by Nagel, Stein, Wainger \cite{NSW}. In Section \ref{S:regpd} I introduce the regularized pseudo-distance, and discuss some of its key properties. Section \ref{S:cgl} covers some basic mean-value formulas first established in \cite{CGL}, and further exploited, among several other works, in \cite{CDGtorino} and \cite{CDGmz}. In Proposition \ref{P:smv} below I reformulate such formulas in terms of the intrinsic pseudo-distance \ref{rhoxy}. In Proposition \ref{P:BP} I show how such intrinsic mean-value formulas lead in a natural way to a potential-theoretic definition of the sub-Laplacean \eqref{LL} above which is akin to the classical approach based on the Blaschke-Privalov Laplacean, see e.g. \cite{DP}. In Section \ref{S:energy} I use the intrinsic mean-value operator to obtain an improved version of the Caccioppoli inequality in \cite{CGL} and also \cite{CDGcpde}. 

In Section \ref{S:fsl}, I introduce the notion of fractional sub-Laplacean $(-\LL)^s$, $0<s<1$ and discuss the extension problem for such nonlocal operator. Although our approach is classical, and goes back to the fundamental ideas of Bochner, and the subsequent work of Balakrishnan \cite{B}, our results are new and it is likely they will find application to other interesting situations. One should see, in this connection, the independent works by Nystr\"om and Sande in  \cite{NS} and by Stinga and Torrea in \cite{ST}, where the case of the standard heat equation is worked out. Also, the recent papers \cite{BG} and \cite{GFT} contain several computations which are quite relevant to the present note. Another relevant work is that of Ferrari and Franchi \cite{FF}, where the authors study fractional powers of sub-Laplaceans in Carnot groups taking as a starting point Folland's definition in \cite{F}. Our Section \ref{S:fsl} generalizes their results. At the onset, our definition of the fractional sub-Laplacean $(-\LL)^s$ in \eqref{flheat} below (based on Balakrishnan's formula) is seemingly different from that in \cite{FF}, based on the Folland's Riesz kernels in \cite{F}. However, in Lemma \ref{L:yes} we recognize that, in fact, in a Carnot group the two definitions are the same. A more substantial difference is that the work \cite{FF} relies on many explicit computations which are not possible in our general setting. In Section \ref{S:parext} I introduce the extension operator $\HH_a$ for the fractional powers of the heat operator $(\p_t - \HH)^s$, see \eqref{extFrHeat} below, with the intent of constructing its Poisson kernel, see Definition \ref{D:extheatpoisson}. Proposition \ref{P:eqheatPK} contains an important property of the latter. Finally, in Section \ref{S:extpb} I solve the extension problem (see Definition \ref{D:ep}) for the operator $(-\LL)^s$. One essential tool is the Poisson kernel, which I construct using its parabolic counterpart. The main results in this section are Propositions \ref{P:UsolvesextL2}, \ref{P:Usolvesext} and \ref{P:trace}.  

While most of the results in the present paper (with the exclusion of Sections \ref{S:fsl}, \ref{S:parext} and  \ref{S:extpb}, which are nonlocal in nature) are of a purely local nature and no geometry is involved, it is nonetheless interesting to study to which extent they continue to hold globally in the presence of suitable curvature assumptions. For instance, one could think that \eqref{LL} above is a diffusion operator on a sub-Riemannian manifold and that a suitable Ricci lower bound condition is assumed in the form of those introduced in \cite{BGjems}. We plan to come back to some of these challenging aspects in the future.

\medskip

\noindent \textbf{Acknowledgment:} I would like to thank Agnid Banrjee, Isidro Munive, Duy-Minh Nhieu and Giulio Tralli for their interest in the present note and for many stimulating discussions. In particular, A. Banerjee has kindly helped with part of the proof of Proposition 8.16. I also thank Bruno Franchi for some interesting feedback.  

\section{The size of the metric balls}\label{S:nsw}

In $\Rn$ with $n\ge 3$ we consider a family of $C^\infty$ vector fields $\mathscr X = \{X_1,...,X_m\}$ satisfying H\"ormander's finite rank assumption in $\Rn$
\[
\operatorname{rank}\operatorname{Lie}[X_1,...,X_m](x) = n,
\]
at every $x\in \Rn$. This condition means that at every point of $\Rn$ the vector fields and a sufficiently large number of their commutators 
\[
X_{j_1}, [X_{j_1},X_{j_2}], [X_{j_1},[X_{j_2},X_{j_3}]],...,[X_{j_1},[X_{j_2},[X_{j_3},...,X_{j_k}]]],...,\ \ \ j_i = 1,...,m,
\]
generate the whole of $\Rn$, i.e., the tangent space. In other words, at every point of $\Rn$ among such differential operators there exist $N$ which are linearly independent. 
Following \cite{NSW} we denote by $Y_1,...,Y_\ell$ the collection of the $X_j$'s and of those commutators which are needed to generate $\Rn$. A ``degree" is assigned to each $Y_i$, namely the corresponding order of the commutator. If $I = (i_1,...,i_n), 1\leq i_j\leq \ell$, is a $n$-tuple of integers,  one defines 
\[
d(I) = \sum_{j=1}^n \text{deg}(Y_{i_j}),\ \ \ \ \text{and}\ \  a_I(x) = \text{det}[Y_{i_1},...,Y_{i_n}]. 
\]

\begin{dfn}\label{D:nswpol}
The \emph{Nagel-Stein-Wainger polynomial} based at a point $x\in \Rn$ is defined by
\begin{equation*}
\Lambda(x,r) = \sum_I\ |a_I(x)| r^{d(I)}, \quad\quad\quad\quad\quad\quad r > 0.
\end{equation*}
\end{dfn}
For a given bounded open set $U\subset \Rn$, we let
\begin{equation}\label{Q}
Q = \text{sup}\ \{d(I)\mid \ |a_I(x)| \ne 0,\ x\in U\},\quad\quad Q(x) = \text{inf}\ \{d(I)\mid |a_I(x)| \ne 0,\ x\in U\},
\end{equation}
and notice that from the work in \cite{NSW} we know 
\begin{equation}\label{Qs}
3 \leq n \leq Q(x) \leq Q.
\end{equation} 
We respectively call the numbers $Q$ and $Q(x)$ the \emph{homogeneous dimension} of $\mathscr X$ relative to $U$, and the \emph{pointwise homogeneous dimension} of $\mathscr X$ at $x$ relative to $U$.
From Definition \ref{D:nswpol}, \eqref{Q} and \eqref{Qs}, it is clear that for every $x\in \Rn$ and $r>0$ we can write
\begin{equation}\label{nswpol}
\Lambda(x,r) = a_{Q(x)} r^{Q(x)} + ... + a_Q r^{Q}.
\end{equation}

We next recall the notion of control, or Carnot-Carath\'eodory distance associated with $\mathscr X$, see \cite{NSW}. A piecewise $C^1$ curve $\gamma:[0,T]\to \Rn$ is called \emph{subunitary} if there exist piecewise continuous functions $a_i:[0,T]\to \Rn$ with $\sum_{i=1}^m |a_i| \le 1$ such that
\[
\gamma'(t)=\sum_{i=1}^m a_i(t)X_i(\gamma(t)),
 \]
whenever $\gamma'(t)$ is defined.
We define the \emph{subunitary length} of $\gamma$ as $\ell_\mathscr X(\gamma) = T$. Given two points $x, y \in \Rn$ denote by  $\mathscr S(x,y)$ the collection of all subunitary curves $\gamma:[0,T]\to \Rn$ such that $\gamma(0)=x$ and $\gamma(T)=y$. By the theorem of Chow-Rashevsky we know that $\mathscr S(x,y)\not= \varnothing$ for every $x, y\in \Rn$. We define the control distance as follows
\[
d(x,y)=\underset{\gamma\in \mathscr S(x,y)}{\inf} \ell_\mathscr X(\gamma).
\]
It is well-known that $d(x,y)$ is an actual distance. The metric ball centered at $x$ with radius $r>0$ will be denoted by $B(x,r) = \{y\in \Rn\mid d(y,x)<r\}$. 
One of the fundamental results in \cite{NSW} is the following.

\begin{thrm}[Size of the metric balls]\label{T:nsw}
Given a bounded set $U\subset \Rn$, there exist $C = C(U,\mathscr X)>0$ and $R_0 = R_0(U,\mathscr X)>0$ such that for every $x\in U$ and $0<r<R_0$ one has
\[
C \La(x,r) \le |B(x,r)| \le C^{-1} \La(x,r).
\]
In particular, there exists $C_d = C_d(U,\mathscr X)>0$ such that for every $x\in U$ and $0<r<R_0/2$
\[
|B(x,2r)| \le C_d |B(x,r)|.
\]
\end{thrm}

We list for future use the following well-known consequence of the doubling condition in Theorem \ref{T:nsw}.

\begin{cor}\label{C:dc}
Given a bounded set $U\subset \Rn$, there exist $C = C(U,\mathscr X)>0$ and $R_0 = R_0(U,\mathscr X)>0$ such that, with
\[
Q = \log_2 C_d,
\]
one has for every $x\in U$ and any $0<r<R\le R_0$ 
\begin{equation}\label{dcallscales}
|B(x,R)| \le C_d \left(\frac Rr\right)^Q |B(x,r)|.
\end{equation}
\end{cor}


\section{Size of the fundamental solution of a sub-Laplacean}\label{S:fs}

Throughout this note we will use the notation
\[
\x u = (X_1 u,...,X_m u)
\] 
to indicate the degenerate gradient of a function $u$ with respect to the family $\mathscr X$. We let
\[
|\x u|^2 = \sum_{i=1}^m (X_i u)^2.
\]
 
Consider now the \emph{sub-Laplacean} associated with such family of vector fields
\[
\mathscr L = - \sum_{i=1}^m X^*_iX_i.
\]
According to Hormander's theorem in \cite{H} the operator $\mathscr L$ is hypoelliptic, i.e., distributional solutions of $\mathscr L u = f$ are $C^\infty$ wherever such is $f$.

Denote by $\Gamma(x,y) = \Gamma(y,x)$ a positive
fundamental solution of $-\mathscr L$ in $\Rn$. We clearly have $\Gamma(x,\cdot)\in C^\infty(\Rn\setminus\{x\})$. The following size estimates of $\Gamma$ were obtained independently by A. Sanchez-Calle \cite{SC}, and by Nagel, Stein and Wainger \cite{NSW}.
\medskip
\begin{thrm}\label{T:estimatesFS}
 Given a bounded set $U\subset \Rn$, there exists $R_0 = R_0(U,\mathscr X)>0$, such that for $x\in U,\ 0<d(x,y)\leq R_0$, one has for $s\in \mathbb{N}\cup\{0\}$, and for some constant $C=C(U,\mathscr X, s) >0$
\begin{align}\label{gradgamma}
& |X_{j_1}X_{j_2}...X_{j_s}\Gamma(x,y)| \leq C^{-1} \frac{d(x,y)^{2-s}}{|B(x,d(x,y))|},\ \ \ \ \ \ \Gamma(x,y) \geq  C \frac{d(x,y)^2}{|B(x,d(x,y))|}.
\end{align}
In the first inequality in \eqref{gradgamma}, one has $j_i\in \{1,...,m\}$ for $i=1,...,s$, and $X_{j_i}$ is allowed to act on either $x$ or $y$.
\end{thrm} 


\section{The regularized pseudo-distance}\label{S:regpd}

Next, we want to express the estimates \eqref{gradgamma} in a more intrinsic fashion.

\begin{dfn}\label{D:mod}
For every $x\in \Rn$ we introduce the \emph{modified polynomial of Nagel, Stein and Wainger} as the function $r\to E(x,r)$ defined by
\[
E(x,r) = \frac{\La(x,r)}{r^2}.
\]
\end{dfn}

The connection between the pointwise homogeneous dimension at $x$ and the asymptotic behavior of $E(x,\cdot)$ is expressed by the following result.

\begin{lemma}\label{L:E}
For any $x\in \Rn$ one has
\begin{equation}\label{quack}
\underset{r\to 0^+}{\lim} \frac{\log E(x,r)}{\log r}  = Q(x) - 2.
\end{equation}
\end{lemma}

\begin{proof}
We notice that de l'Hospital rule gives
\[
\underset{r\to 0^+}{\lim} \frac{\log E(x,r)}{\log r}  = \underset{r\to 0^+}{\lim} \frac{r E'(x,r)}{E(x,r)}.
\]
The claim \eqref{quack} is now easily obtained by this observation and by \eqref{nswpol}, which gives
\begin{align*}
\underset{r\to 0^+}{\lim} \frac{r E'(x,r)}{E(x,r)} & = (Q(x) - 2) \underset{r\to 0^+}{\lim} \frac{a_{Q(x)} + ... + (Q-2)/(Q(x)-2) a_Q r^{Q-Q(x)}}{a_{Q(x)} + ... + a_Q r^{Q-Q(x)}} = Q(x) - 2.
\end{align*}

\end{proof}

From \eqref{Q}, \eqref{Qs} and \eqref{nswpol} we also obtain the following simple, yet important property.

\begin{lemma}\label{L:simple}
Given a bounded set $U\subset \Rn$, there exist $C_2, R_0>0$, depending on $U$ and $\mathscr X$, such that for every $x\in U$ and $0<r<R_0$ one has
\[
C_2 \le \frac{r E'(x,r)}{E(x,r)} \le C^{-1}_2.
\]
\end{lemma}

It is clear from \eqref{Qs} and \eqref{nswpol} that $E(x,\cdot)$ is strictly increasing, and therefore it is invertible on its domain. We denote its inverse by
\[
F(x,\cdot) = E(x,\cdot)^{-1},
\]
so that
\[
F(x,E(x,r)) = r,\ \ \ \ \ \ \ \ \ E(x,F(x,r)) = r.
\]
Using the function $E(x,\cdot)$ we can express the size estimate for $\Gamma(x,y)$ in \eqref{gradgamma} in the following way
\begin{equation}\label{segamma}
 \frac{C}{E(x,d(x,y))} \le \Gamma(x,y) \le \frac{C^{-1}}{E(x,d(x,y))}.
 \end{equation} 

\begin{dfn}\label{D:reg}
For a fixed point $x\in \Rn$ we define the \emph{regularized pseudo-distance} centered at $x$ as
\begin{equation}\label{rhoxy}
\rho_x(y) = \begin{cases}
F(x,\Gamma(x,y)^{-1}),\ \ \ \ \ y \not= x,
\\
0\ \ \ \ \ \ \ \ \ \ \ \ \ \ \ \ \ \ \ \ \ \ \ y = x.
\end{cases}
\end{equation}
\end{dfn}

It is worth observing explicitly that applying the function $E(x,\cdot)$ to both sides of \eqref{rhoxy} we obtain for any $y\not= x$
\begin{equation}\label{gammaE}
\Gamma(x,y) = \frac{1}{E(x,\rho_x(y))}.
\end{equation}

\begin{prop}\label{P:goodrho}
One has $\rho_x\in C^\infty(\Rn\setminus\{x\})\cap C(\Rn)$. Moreover, given a bounded set $U\subset \Rn$, there exist positive numbers $C, R_0$, and $a\geq 1$, depending on $U$ and $\mathscr X$, such that for every $x\in U$, and every $y\in B(x,R_0)$, one has
\begin{equation}\label{F2}
a^{-1} d(x,y) \leq  \rho_x(y) \leq  a d(x,y),
\end{equation}
\begin{equation}\label{F3}
|\x \rho_x(y)| \leq C.
\end{equation}
\end{prop}

\begin{proof}
Since $r \to E(x,r)$ is a polynomial function with positive coefficients, we infer that $t\to F(x,t)$ belongs to $C^\infty(\mathbb R)$. It is then clear that $\rho_x \in C^\infty(\Rn\setminus \{x\})$.
Keeping \eqref{segamma} in mind, and that $E(x,0) = F(x,0) = 0$, 
we see that $\Gamma(x,y) \to + \infty$ as $y \to x$. As a consequence, $\rho_x \in C(\Rn)$.
If we write \eqref{segamma} as follows
\[
\frac{C}{\Gamma(x,y)} \leq E(x,d(x,y)) \leq \frac{C^{-1}}{\Gamma(x,y)},
\]
and we apply the function $F(x,\cdot)$ to this inequality, we obtain
\[
F\left(x,\frac{C}{\Gamma(x,y)}\right)\ \leq\ d(x,y)\ \leq\ F\left(x,\frac{C^{-1}}{\Gamma(x,y)}\right)\ .
\]
From the latter equation, and from the doubling properties of the function $r\to F(x,r)$, we now obtain \eqref{F2}. We next prove \eqref{F3}.
The chain rule and the inverse function theorem give for $y\not= x$
\begin{equation}\label{gradrho}
\x \rho_x(y) = - \frac{F'(x,\Gamma(x,y)^{-1})}{\Gamma(x,y)^2} \x\Gamma(x,y) = - \frac{1}{E'(x,\rho_x(y)) \Gamma(x,y)^2} \x\Gamma(x,y).
\end{equation}
Substitution of \eqref{gammaE} in \eqref{gradrho} allows to rewrite the latter equation in the more suggestive way
\begin{equation}\label{zerodiv}
\x\rho_x(y)  = - \frac{E(x,\rho_x(y))^2}{E'(x,\rho_x(y))} \x \Gamma(x,y).
\end{equation}
Using \eqref{gradgamma} we obtain
\[
|\x \Gamma(x,y)| \leq  \frac{C}{\rho_x(y) E(x,\rho_x(y))}.
\]
Substituting this information in \eqref{zerodiv} we find
\[
|\x \rho_x(y)| \leq C \frac{E(x,\rho_x(y))}{\rho_x(y) E'(x,\rho_x(y))}.
\]
The desired estimate \eqref{F3} now follows from Lemma \ref{L:simple}.

\end{proof}


\section{Mean-value formulas for sub-Laplaceans}\label{S:cgl}

We next recall some mean-value formulas that were found in \cite{CGL}.
For every $t>0$ we denote by
\begin{equation}\label{sls}
\Om(x,t) = \left\{y\in \Rn\mid
\Gamma(x,y)>\frac{1}{t}\right\}
\end{equation}
the superlevel set of $\Gamma(x,\cdot)$. The
following basic result was proved in \cite{CGL}. 

\begin{prop}\label{P:mv}
For any $\psi \in C^{\infty}(\Rn)$, $x\in \Rn$ and $t>0$ one has
\begin{equation}\label{psi}
\psi(x) = \int_{\partial \Om(x,t)} \psi(y) \frac{|\x
\Gamma(x,y)|^2}{|\nabla \Gamma(x,y)|} dH_{n-1}(y) -
\int_{\Om(x,t)} \mathscr L \psi(y)
\left[\Gamma(x,y)-\frac{1}{t}\right] dy,
\end{equation}
where $H_{n-1}$ denotes the standard $(n-1)$-dimensional Hausdorff measure in $\Rn$. 
\end{prop}

We intend to formulate Proposition \ref{P:mv} in a more intrinsic fashion. With this objective in mind we introduce the following notion.

\begin{dfn}\label{D:xball}
We define the $\xx$-\emph{ball} centered at $x$ with radius $r>0$ as the set
\[
B_\xx(x,r) = \Om(x,E(x,r)) = \left\{y\in \Rn\mid \Gamma(x,y) > \frac{1}{E(x,r)}\right\}.
\]
We note explicitly that in view of \eqref{rhoxy} we can rewrite
\[
B_\xx(x,r) = \{y\in \Rn\mid \rho_x(y) < r\}.
\]
\end{dfn}

From formula \eqref{F2} in Proposition \ref{P:goodrho} we immediately obtain that for every bounded set $U\subset \Rn$ there exist $a\ge1$  and $R_0>0$, depending on $U$ and $\mathscr X$, such that for every $x\in U$ and $0<r<R_0$ one has with the number $a>0$ as in \eqref{F2}
\begin{equation}\label{bb}
B(x,a^{-1}r) \subset B_\xx(x,r) \subset B(x,ar).
\end{equation}
Combining Theorem \ref{T:nsw} with \eqref{nswpol} and \eqref{bb}, we conclude that for every $x\in U$ and $0<r<R_0$ one has
\begin{equation}\label{lampadino}
C \La(x,r) \le |B_\xx(x,r)| \le C^{-1} \La(x,r).
\end{equation}
The estimate \eqref{lampadino} and the expression \eqref{nswpol} show, in particular, that for any fixed $x\in \R^N$ and every $\alpha<Q(x)$, one has
\begin{equation}\label{asy}
\underset{r\to 0^+}{\lim} \frac{|B_\xx(x,r)|}{r^\alpha} = 0. 
\end{equation}

Our next objective is to express the mean-value formula \eqref{psi} in Proposition \ref{P:mv} in a more intrinsic fashion using the regularized pseudo-distance $\rho_x$ and the $\xx$-balls $B_\xx(x,r)$. With this goal in mind we notice that the inverse function theorem gives
\begin{equation}\label{F'}
F'(x,E(x,r)) = \frac{1}{E'(x,r)}.
\end{equation}
We thus have from \eqref{F'}
\begin{equation}\label{F'rho}
F'(x,\Gamma(x,y)^{-1}) = \frac{1}{E'(x,\rho_x(y))}.
\end{equation}
The chain rule now gives
\[
\nabla \rho_x(y) = - F'(x,\Gamma(x,y)^{-1}) \Gamma(x,y)^{-2} \nabla \Gamma(x,y),
\]
and similarly
\[
\x \rho_x(y) = - F'(x,\Gamma(x,y)^{-1}) \Gamma(x,y)^{-2} \x \Gamma(x,y).
\]
Combining the latter two equations with \eqref{rhoxy} and \eqref{F'rho}, we find
\[
\frac{|\x
\Gamma(x,y)|^2}{|\nabla \Gamma(x,y)|} = 
\frac{\Gamma(x,y)^2}{F'(x,\Gamma(x,y)^{-1})} \frac{|\x \rho_x(y)|^2}{|\nabla \rho_x(y)|} = \frac{E'(x,\rho_x(y))}{E(x,\rho_x(y))^2} \frac{|\x \rho_x(y)|^2}{|\nabla \rho_x(y)|} 
\]

\begin{dfn}\label{D:smv}
We define the \emph{surface mean-value operator} acting on a function $\psi\in C(\Rn)$ as follows
\begin{align}\label{smv}
\mathscr M_\xx \psi(x,r) & =  \frac{E'(x,r)}{E(x,r)^2} \int_{\partial B_\xx(x,r)} \psi(y) \frac{|\x \rho_x(y)|^2}{|\nabla \rho_x(y)|}  dH_{n-1}(y).
\notag
\end{align}
\end{dfn}

Using Definition \ref{D:smv} we can reformulate \eqref{psi} in the following suggestive way.

\begin{prop}\label{P:smv}
Let $\psi \in C^{2}(\Rn)$. For any $x\in \R^n$ and
$r>0$ one has
\begin{equation}\label{psi2}
\mathscr M_\xx \psi(x,r) = \psi(x) +
\int_{B_\xx(x,r)} \mathscr L \psi(y)
\left[\Gamma(x,y)-\frac{1}{E(x,r)}\right] dy.
\end{equation}
In particular, letting $\psi \equiv 1$ in \eqref{psi2}, we find
\begin{equation}\label{psi3}
 \int_{\partial B_\xx(x,r)} \frac{|\x \rho_x(y)|^2}{|\nabla \rho_x(y)|}  dH_{N-1}(y) = \frac{E(x,r)^2}{E'(x,r)},
\end{equation}
for every $r>0$.
\end{prop}

We next show how Proposition \ref{P:smv} can be used to introduce a subelliptic version of the Blaschke-Privalov Laplacean from classical potential theory. We recall that if $\psi\in C^2(\Rn)$, and we denote with $\Delta \psi = \sum_{k=1}^n \frac{\p^2 \psi}{\p x_k^2}$ the standard Laplacean, then for every $x\in \Rn$ one has
\begin{equation}\label{BP0}
\Delta \psi (x) = 2n\ \underset{r\to 0}{\lim} \frac{\mathscr M \psi(x,r) - \psi(x)}{r^2},
\end{equation}
where we have indicated with
\[
\mathscr M \psi(x,r) = \frac{1}{\sigma_{n-1} r^{n-1}} \int_{S(x,r)} \psi(y) d\sigma(y),
\]
the classical spherical mean-value operator acting on $\psi$. We want to show next that a similar formula holds for the subelliptic mean-value operator $\mathscr M_\xx \psi(x,r)$. With this objective in mind we introduce a crucial definition.

\begin{dfn}\label{D:critical2}
For a given $x\in \Rn$ and $r>0$ we define the \emph{density function at} $x$ by the formula
\[
\zeta(x,r) \overset{def}{=} \int_0^r \frac{E'(x,t)|B_\xx(x,t)|}{E^2(x,t)} dt.
\]
\end{dfn}

The motivation for Definition \ref{D:critical2} will be clear from the statement of Proposition \ref{P:BP}, and from its proof. Before proceeding, we pause to note the following interesting fact.

\begin{prop}\label{P:critical}
Let $\Ga$ be a Carnot group. Then, there exists a universal constant $\alpha = \alpha(\Ga)>0$ such that for every $x\in \Ga$ and every $r>0$ one has
\[
\zeta(x,r) = \alpha r^2.
\]
\end{prop}

\begin{proof}
We notice that in a Carnot group the Nagel-Stein-Wainger polynomial is actually a monomial which is independent of $x\in \Ga$, i.e., $\Lambda(x,r) = \omega r^Q$, where $\omega = \omega(\Ga)>0$ is a universal constant, and $Q$ is the homogeneous dimension of $\Ga$. Consequently, one has $E(x,r) = \omega r^{Q-2}$. Since the fundamental solution of any sub-Laplacian is homogeneous of degree $2-Q$ (see Theorem 2.1 in Folland's seminal paper \cite{F}), and invariant with respect to left-translations, we see that $|B_\xx(x,r)| = \beta r^Q$ for every $x\in \Ga$ and $r>0$, where $\beta = \beta(\Ga)>0$ is a universal constant. We infer that for every $x\in \Ga$ and $t>0$ one has
\[
\frac{E'(x,t) |B_\xx(x,t)|}{E(x,t)^2} = \frac{(Q-2) \omega \beta t^{2Q-3}}{\omega^2 t^{2Q-4}} = (Q-2) \omega^{-1} \beta t.
\]
The desired conclusion follows immediately from this formula and the definition of $\zeta(x,r)$, if we set $\alpha = (Q-2) \omega^{-1} \beta/2$.

\end{proof}

Although in the general case of a sub-Laplacean in $\Rn$ we do not have a precise formula as in Proposition \ref{P:critical}, the qualitative behavior of $r\to \zeta(x,r)$ is locally uniformly analogous to the case of a Carnot group.

\begin{prop}\label{P:quali}
Given a bounded set $U\subset \Rn$, there exist $\alpha, R_0>0$, depending on $U$ and $\mathscr X$, such that for every $x\in U$ and $0<r<R_0$ one has
\[
\alpha r^2 \le \zeta(x,r) \le \alpha^{-1} r^2.
\]
\end{prop}

\begin{proof}
We write
\[
\frac{E'(x,t)|B_\xx(x,t)|}{E^2(x,t)} =  \frac{t E'(x,t)}{E(x,t)} \frac{|B_\xx(x,t)|}{\La(x,t)}\ t.
\]
By Lemma \ref{L:simple} and \eqref{lampadino} we conclude that for some constant $\bar C>0$ one has
\[
\bar C t \le \frac{E'(x,t)|B_\xx(x,t)|}{E^2(x,t)} \le \bar C^{-1} t.
\]
The desired conclusion immediately follows upon integrating the above inequalities on $(0,r)$.

\end{proof}
 
The main motivation for introducing Definition \ref{D:critical2}
is the following result.

\begin{prop}[Blaschke-Privalov sub-Laplacean]\label{P:BP}
Let $\psi \in C^2(\Rn)$. Then, for any $x\in \Rn$ one has
\begin{equation}\label{psi20}
\underset{r\to 0^+}{\lim} \frac{\mathscr M_\xx \psi(x,r) - \psi(x)}{\zeta(x,r)} = \mathscr L \psi(x).
\end{equation}
\end{prop}

\begin{proof}
By means of \eqref{psi2}, de l'Hospital rule and the coarea formula, we find
\begin{align*}
& \underset{r\to 0^+}{\lim} \frac{\mathscr M_\xx \psi(x,r) - \psi(x)}{\zeta(x,r)} = \underset{r\to 0^+}{\lim} \frac{\frac{d}{dr} \int_{B_\xx(x,r)} \mathscr L \psi(y)
\left[\Gamma(x,y)-\frac{1}{E(x,r)}\right] dy}{\zeta'(x,r)}
\\
& =  \underset{r\to 0^+}{\lim} \frac{\int_{\p B_\xx(x,r)} \frac{\mathscr L \psi(y)}{|\nabla \rho_x(y)|}
\left[\Gamma(x,y)-\frac{1}{E(x,r)}\right] dy - \frac{d}{dr} E(x,r)^{-1}\int_{B_\xx(x,r)} \mathscr L \psi(y) dy}{\zeta'(x,r)}
\\
& = \underset{r\to 0^+}{\lim} \frac{E'(x,r)|B_\xx(x,r)|}{\zeta'(x,r) E(x,r)^2} \frac{1}{|B_\xx(x,r)|}\int_{B_\xx(x,r)} \mathscr L \psi(y) dy
\\
& = \underset{r\to 0^+}{\lim} \frac{E'(x,r)|B_\xx(x,r)|}{\zeta'(x,r) E(x,r)^2}\ \underset{r\to 0^+}{\lim} \frac{1}{|B_\xx(x,r)|}\int_{B_\xx(x,r)} \mathscr L \psi(y) dy
\\
& = \mathscr L \psi(x)\ \underset{r\to 0^+}{\lim} \frac{E'(x,r)|B_\xx(x,r)|}{\zeta'(x,r)  E(x,r)^2},
\end{align*}
where in the last equality we have used the fact that $\mathscr L \psi \in C(\Rn)$, and that from \eqref{asy} we know that $|B_\xx(x,r)| \to 0$ as $r\to 0^+$. Since Definition \ref{D:critical2} gives
\[
\frac{E'(x,r)|B_\xx(x,r)|}{\zeta'(x,r)  E(x,r)^2} \equiv 1,
\]
the desired conclusion immediately follows.

\end{proof}

\begin{rmrk}\label{R:density}
We note here that the above proof of Proposition \ref{P:BP}, based on a simple application of de L'Hospital rule, leads in a natural way to our Definition \ref{D:critical2} of the density function $\zeta(x,r)$. We mention in this connection that, although we were not aware of this at the time we wrote  a first draft of this note, Proposition \ref{P:BP} has already appeared in the literature in Proposition 3.5 in the interesting paper \cite{BL}. To see this, we observe that in \cite{BL} the authors base their entire analysis on formula \eqref{psi} in Proposition \ref{P:mv} above. They thus consider the mean-value operator
\[
m_t(\psi)(x) \overset{def}{=}  \int_{\partial \Om(x,t)} \psi(y) \frac{|\x
\Gamma(x,y)|^2}{|\nabla \Gamma(x,y)|} dH_{n-1}(y)
\]
and their Proposition 3.5 states that
\begin{equation}\label{bl}
\underset{t\to 0^+}{\lim} \frac{m_t(u)(x) - u(x)}{q_t(x)} = \mathscr L\psi(x),
\end{equation} 
where with $\Om(x,t)$ as in \eqref{sls} above, they define
\[
q_t(x) = \int_{\Om(x,t)} \left[\Gamma(x,y) - \frac 1t\right] dy.
\]
Using the coarea formula they subsequently recognize in their formula (11.23) the following alternative expression 
\[
q_t(x) = \int_0^t \frac{|\Om(x,s)|}{s^2} ds.
\]
Now, making the change of variable $s = E(x,t)$ in our Definition \ref{D:critical2} we have $ds = E'(x,t) dt$, and thus we find from Definition \ref{D:xball}
\[
\zeta(x,r) = \int_0^{E(x,r)} \frac{|B_\xx(x,F(x,s))|}{s^2} ds = \int_0^{E(x,r)} \frac{|\Om(x,s)|}{s^2} ds.
\]
From these observations it is thus clear that, up to the non-isotropic ``rescaling" $r\to E(x,r)$, our density function $\zeta(x,r)$ is precisely the function $q_r(x)$ in \cite{BL} since we have
\[
q_r(x) = \zeta(x,F(x,r)).
\]
In particular, keeping Proposition \ref{P:critical} in mind we see that in a Carnot group one has 
\[
q_r(x,r) = \gamma(\Ga) r^{2/(Q-2)},
\]
where $\gamma(\Ga)>0$ is a universal constant.

\end{rmrk}

Combining Propositions \ref{P:critical} and \ref{P:BP} we obtain the following interesting result which parallels the classical Blaschke-Privalov formula \eqref{BP0} for the Laplacean.

\begin{prop}\label{P:BPgroup}
Let $\Ga$ be a Carnot group. Then, there exists a universal constant $\alpha = \alpha(\Ga)>0$ such that for every $x\in \Ga$ and every $\psi \in C^2(\Ga)$ one has
\begin{equation}\label{psi22}
\underset{r\to 0^+}{\lim} \frac{\mathscr M_\xx \psi(x,r) - \psi(x)}{r^2} = \alpha^{-1} \mathscr L \psi(x).
\end{equation}
\end{prop}


\section{An improved energy estimate}\label{S:energy}

In this section we establish an energy estimate which is reminiscent of the classical Caccioppoli inequality for second-order uniformly elliptic equations, except that in the right-hand side we have a surface integral, instead of a solid one. It is worth noting here that we obtain such energy estimate completely independently from the existence of cut-off functions tailor made on the intrinsic geometry of the metric balls constructed in \cite{GNlip}. 

In what follows we consider a function $\psi\in C^2(\R^N)$. For a given $h\in C^2(\R)$ the chain rule gives
\[
\mathscr L (h\circ \psi) = h''(\psi) |\nabla \psi|^2 + h'(\psi) \mathscr L \psi.
\] 
Applying this identity with $h(t) = t^2$ we find
\[
\mathscr L (\psi^2) = 2 |\x \psi|^2 + 2 \psi \mathscr L \psi.
\] 
Combining this observation with \eqref{psi2} in Proposition \ref{P:smv}, we find
\begin{equation}\label{psi3}
\mathscr M_\xx \psi^2(x,r) = \psi^2(x) +
2 \int_{B_\xx(x,r)} \left(|\x \psi|^2 + 2 \psi \mathscr L \psi\right)
\left[\Gamma(x,y)-\frac{1}{E(x,r)}\right] dy.
\end{equation}
Similarly to the proof of Proposition \ref{P:BP} we now find
from \eqref{psi3}
\begin{equation}\label{psi4}
\frac{\p \mathscr M_\xx \psi^2}{\p r}(x,r) =  \frac{2 E'(x,r)}{E(x,r)^2} \int_{B_\xx(x,r)} \left(|\x \psi|^2 + \psi \mathscr L \psi\right) dy.
\end{equation}
If we suppose that $\psi \mathscr L \psi \ge 0$, then we obtain
\begin{equation}\label{psi5}
\frac{\p \mathscr M_\xx \psi^2}{\p r}(x,r) \ge  \frac{2 E'(x,r)}{E(x,r)^2} \int_{B_\xx(x,r)} |\x \psi|^2  dy.
\end{equation}
Integrating this inequality for $0<s<r<t$, we find
\begin{align*}
\mathscr M_\xx \psi^2(x,t) - \mathscr M_\xx \psi^2(x,s) & \ge \int_s^t \frac{2 E'(x,r)}{E(x,r)^2} \int_{B_\xx(x,r)} |\x \psi|^2  dy dr
\\
& \ge \left(\int_s^t \frac{2 E'(x,r)}{E(x,r)^2}  dr\right) \int_{B_\xx(x,s)} |\x \psi|^2  dy
\\
& \ge \frac{C(t-s)}{E(x,t)} \int_{B_\xx(x,s)} |\x \psi|^2  dy,
\end{align*}
where in the last inequality we have used Lemma \ref{L:simple} and the fact that $r\to E(x,r)$ is increasing. From the latter inequality we obtain the following result.

\begin{prop}[Improved Caccioppoli inequality]\label{P:energy}
Suppose that $\psi \mathscr L \psi \ge 0$. Then, given any bounded set $U\subset \R^N$ there exist constants $C, R_0>0$, depending on $U$ and $\xx$, such that for every $x\in U$ and $0<s<t<R_0$ one has
\[
\int_{B_\xx(x,s)} |\x \psi|^2  dy \le \frac{C E(x,t)}{t-s} \mathscr M_\xx \psi^2(x,t).
\]
\end{prop}


\section{The fractional sub-Laplacean and its heat counterpart}\label{S:fsl}

In this section given a number $0<s<1$ we lay down the preliminaries of a theory of fractional powers $(-\LL)^s$ of the differential operator $-\LL$ defined in \eqref{LL} above and its associated \emph{heat operator} in $\mathbb R^{n+1}$
\begin{equation}\label{HO}
\mathscr H = \frac{\partial}{\partial t} - \mathscr L.
\end{equation}

By H\"ormander's theorem in \cite{H} the operator $\mathscr H$ is hypoelliptic. The reader should notice here that the existence of a global fundamental solution $p(x,y,t)$ of the operator $\HH$ is not guaranteed without some serious additional assumptions. One way of trivializing the geometry is to assume that, outside of a large compact set, the operator $\mathscr L$ coincides with the standard Laplacian (of course, it is assumed here that the transition from $\LL$ to $\Delta$ occurs smoothly). In this way, all results obtained are of a local nature, if one's focus is primarily in such aspect. This is exactly what we assume in the present section. 

Under such hypothesis $\HH$ admits a positive fundamental solution $p(x,t;\xi,\tau) = p(x,\xi;t-\tau)$ which is smooth in $\Rnn \setminus\{(\xi,\tau)\}$. Clearly, one has 
\begin{equation}\label{Hp}
\mathscr H p(x,\xi;t-\tau) = \frac{\partial}{\partial t} p(x,\xi;t-\tau) - \mathscr L_x p(x,\xi;t-\tau) = 0,\ \ \ \ \ \text{in}\ \Rnn \setminus\{(\xi,\tau)\}.
\end{equation}
The following basic result was established  in Theorem 3 in \cite{JSC} (the reader should note that there is an obvious typo in the right-hand side of the relevant formula in Theorem 3. The term $t^{i+\frac{|I|+|J|}2}$ must be changed into $t^{- i-\frac{|I|+|J|}2}$). One should also see Theorem 4.14 in \cite{KS1}
and Theorem 8.1 in \cite{BBLU}.

\begin{thrm}\label{T:KS}
The fundamental solution $p(x,t;\xi,\tau) = p(x,\xi;t-\tau)$ with singularity at $(\xi,\tau)$ satisfies the following size estimates : there exists $M = M(X)>0$ and for every
$k , s\in \mathbb{N}\cup\{0\}$, there exists a constant $C=C(X, k, s) >0$, such that
\begin{equation}\label{gaussian1}
\bigg|\frac{\partial^k}{\partial t^k} X_{j_1}X_{j_2}...X_{j_s}p(x,t;\xi,\tau)\bigg| \leq \frac{C}{(t - \tau)^{k+ \frac{s}2}} \frac{1}{|B(x,\sqrt{t - \tau})|} \exp \bigg( - \frac{M d(x,\xi)^2}{t - \tau}\bigg),
\end{equation}
\begin{equation}\label{gaussian2}
p(x,t;\xi,\tau) \geq  \frac{C^{-1}}{|B(x,\sqrt{t - \tau})|} \exp \bigg( - \frac{M^{-1} d(x,\xi)^2}{t - \tau}\bigg),
\end{equation}
for every $x,\xi \in \Rn$, and any $-\infty < \tau < t < \infty $.
\end{thrm}

If one is interested instead in the connection between geometry and global estimates of heat kernels on sub-Riemannian manifolds, then one should consult the works \cite{BG09}, \cite{BGjems}, along with the companion papers \cite{BBG}, \cite{BGimrn} and \cite{BBGM}.

The heat semigroup $P_t = e^{t \LL}$ is defined by the following formula
\[
P_t u(x) = \int_{\Rn}  p(x,y,t) u(y) dy,\ \ \ \ \ \ \ \ \ \ u \in \mathscr S(\Rn). 
\]
The semigroup is sub-Markovian, i.e., $P_t 1 \le 1$, and defines a family of bounded operators $P_t : L^2 (\Rn) \rightarrow L^2 (\Rn)$ having the following properties:
\begin{itemize}
\item[(i)] $P_0=\operatorname{Id}$ and for $s,t  \ge 0$, $P_s P_t =P_{s+t}$;
\item[(ii)] for $u \in L^2 (\Rn)$, 
\[
\| P_t u \|_{L^2 (\Rn)} \le \| u \|_{L^2 (\Rn)};
\]
\item[(iii)] for $u  \in L^2 (\Rn)$, the map $ t \to P_t u$ is continuous in $L^2 (\Rn)$;
\item[(iv)] for $u ,v \in L^2 (\Rn)$ one has
\[
\int_{\Rn} (P_t u) v dx= \int_{\Rn} u (P_t v)  dx.
\]
\end{itemize}
Properties (i)-(iv) can be summarized by saying that $\{P_t\}_{t \ge 0}$ is a self-adjoint strongly continuous contraction semigroup on $L^2 (\Rn)$. 
From the spectral decomposition, it is also easily checked that the operator $\LL$ is furthermore the generator of this semigroup, that is for $u \in \mathscr{D}(\LL)$ (the domain of $\LL$),
\begin{equation}\label{ig}
\lim_{t \to 0^+} \left\| \frac{P_t u -u}{t} - \LL u \right\|_{ L^2 (\Rn) }=0.
\end{equation}
This implies that for $t \ge 0$, $P_t \mathscr{D}(\LL) \subset \mathscr{D}(\LL)$, and that for $u \in \mathscr{D}(\LL)$, 
\[
\frac{d}{dt}P_t u = P_t \LL u= \LL P_t u,
\]
the derivative in the left-hand side of the above equality being taken in $ L^2 (\Rn)$. For a construction of the heat semigroup, its main properties and regularity we refer the reader to the forthcoming book \cite{BGbook}.
The identity \eqref{ig} shows in particular that for every $0<b<1$ one has in $L^2(\Rn)$
\begin{equation}\label{ig2}
||P_t u - u||_{L^2(\Rn)} = o(t^b)\ \ \ \ \ \ \ \ \ \ \ \text{as}\ t\to 0^+.
\end{equation}

Under our assumptions the semigroup is \emph{stochastically complete}, i.e., $P_t 1 = 1$. This means that for every $x\in \Rn$, and $t>0$ one has 
\begin{equation}\label{P1}
\int_{\Rn} p(x,y,t) dy = 1.
\end{equation}
For a proof of \eqref{P1} one can see (3.2) in Theorem 3.4 in \cite{BBLU}. In their work the authors treat operators in non-divergence form, but they allow for lower order terms, and thus our situation is included. We note that, notably, \eqref{P1} is verified in a large number of situations in which the geometry becomes relevant. One sufficient condition for stochastic completeness is contained in the following result.

\begin{thrm}\label{T:grigoryan}
Let $M$ be a complete connected Riemannian manifold and denote by $V(x,r) = \operatorname{Vol}(B(x,r))$ the volume of the metric balls. If  for some point $x_0\in M$ one has
\begin{equation}\label{logvol}
\int^\infty \frac{r}{\ln V(x_0,r)} dr = \infty,
\end{equation}
then $M$ is stochastically complete.
\end{thrm}
Theorem \ref{T:grigoryan} was proved by Grigor'yan in 1987, see \cite{Gri}. In 1994 it was generalized by Sturm to the setting of Dirichlet forms on a metric space, see \cite{Sturm}. A version for sub-Riemannian spaces was established by Munive in \cite{Munive}. 

Before proceeding we pause to establish a useful lemma. In such lemma we assume that the doubling condition for the volume of the metric balls, and therefore the ensuing \eqref{dcallscales}, be valid on the whole space. Under our hypothesis this is guaranteed by (2.8) in Proposition 2.5 in \cite{BBLU} (the reader should bear in mind that we are assuming in this section that outside a large compact set $\LL$ is the standard Laplacean). 

\begin{lemma}\label{L:exp}
For any given $\alpha, \beta>0$ there exists a constant $C>0$ depending on $C_d$ and $\alpha, \beta$, such that 
\begin{equation}\label{exp}
\int_{\Rn} d(x,y)^\beta \exp\left(-\alpha \frac{d(x,y)^2}{t}\right) dy \le C t^\frac{\beta}2 |B(x,\sqrt t)|.
\end{equation}
\end{lemma}

\begin{proof}
We write
\begin{align*}
& \int_{\Rn} d(x,y)^\beta \exp\left(-\alpha \frac{d(x,y)^2}{t}\right) dy = \int_{d(y,x)<\sqrt t} d(x,y)^\beta \exp\left(-\alpha \frac{d(x,y)^2}{t}\right) dy 
\\
& + \sum_{k=0}^\infty \int_{2^k \sqrt t \le d(y,x)<2^{k+1} \sqrt t} d(x,y)^\beta \exp\left(-\alpha \frac{d(x,y)^2}{t}\right) dy
\\
& \le t^\frac{\beta}2 |B(x,\sqrt t)| + \sum_{k=0}^\infty (2^{k+1} \sqrt t)^\beta \exp\left(-\alpha \frac{(2^k \sqrt t)^2}{t}\right) |B(x,2^{k+1} \sqrt t)|.
\end{align*}
Using \eqref{dcallscales} we find
\[
|B(x,2^{k+1} \sqrt t)| \le C_d 2^{Q(k+1)} |B(x,\sqrt t)|.
\]
Substitution in the above inequality gives the desired conclusion \eqref{exp}.

\end{proof}

After these preliminaries, we are now ready to move to the core part of this section. Using the semigroup $P_t = e^{t \LL}$ it is natural to propose the following definition for the fractional powers of the operator $\LL$. 

\begin{dfn}\label{D:flheat}
Let $0<s<1$. For any $u\in \mathscr S(\Rn)$ we define the nonlocal operator
\begin{align}\label{flheat}
(-\LL)^s u(x) & =   \frac{1}{\G(-s)} \int_0^\infty t^{-s-1} \left[P_t u(x) - u(x)\right] dt
\\
& = - \frac{s}{\G(1-s)} \int_0^\infty t^{-s-1} \left[P_t u(x) - u(x)\right] dt.
\notag
\end{align}
\end{dfn}

In an abstract setting, formula \eqref{flheat} is due to Balakrishnan, see \cite{Bthesis} and \cite{B}. One should also see IX.11 in \cite{Y}, in particular formulas (4) and (5) on p. 260 and their ensuing discussion, and (5.84) on p. 120 in \cite{SKM}. The integral defining the operator in the right-hand side of \eqref{flheat} must be interpreted as a Bochner integral in $L^2(\Rn)$. We note explicitly that, in view of \eqref{ig2} and of (ii) above, the integral is convergent (in $L^2(\Rn)$) for every $u\in \mathscr D(\LL)$, and thus in particular for every $u\in \mathscr S(\Rn)$.

In the special setting of Carnot groups a seemingly different definition of fractional sub-Laplacean in a Carnot group, based on the Riesz kernels, was set forth in the work \cite{FF}. Their starting point is the classical observation that
\[
(-\LL)^s u = (-\LL)^{s-1+1} u = (-\LL)^{s-1}(-\LL u).
\]
Since now $s-1<0$, one can use Folland's \emph{Riesz kernels} $R_\beta$, which he proved in \cite{F} provide the negative powers of $-\LL$. Here, if $Q$ is the homogeneous dimension of the group $\mathbb G$ associated with the anisotropic, and $0<\beta<Q$, then the Riesz kernels are defined by
\[
R_\beta(x) = \frac{1}{\G(\beta/2)} \int_{0}^\infty t^{\frac \beta{2}} p(x,t) \frac{dt}t,
\] 
where $p(x,t)$ is the heat kernel in $\mathbb G$. For instance, when $\mathbb G = \Rn$ is Abelian, one easily recognizes that $R_\beta(x) = c(n,\beta) |x|^{\beta - n}$. The fractional integration operator of order $\beta$ is defined in \cite{F} as
\[
I_\beta(f) = f \star R_\beta,
\]
where $\star$ indicates the group convolution defined by $f\star g(x) = \int_{\mathbb G} f(y) g(y^{-1} \circ x) dy$, with $\circ$ indicating the group multiplication. It was proved in \cite{F} that $I_\beta = (-\LL)^{-\beta/2}$. Given these notations, the definition of fractional sub-Laplacean in \cite{FF} is (see (ii) in Proposition 3.3) 
\begin{equation}\label{FF}
(-\LL)^s u = (-\LL u) \star R_{2-2s} = I_{2-2s}(-\LL u).
\end{equation}

In the case $\mathbb G = \Rn$ one recognizes that all the various notions of fractional Laplacean coincide,  but even in the classical setting such task in not altogether trivial. For this aspect we refer the reader to \cite{Kwa} and \cite{FF}. A natural question to ask is whether, at least in the setting of a Carnot group, our Definition \ref{D:flheat} coincides with \eqref{FF}. As we next show, the answer is yes (see also Remark \ref{R:PK} below). 

\begin{lemma}\label{L:yes}
Let $\mathbb G$ be a Carnot group and $u\in \mathscr D(\mathbb G)$. Then, Definition \ref{D:flheat} coincides with \eqref{FF}.
\end{lemma}

\begin{proof}
We have from \eqref{FF}
\begin{align*}
(-\LL)^s u(x) & = (-\LL u) \star R_{2-2s}(x)   = - \int_{\mathbb G} \LL u(y) R_{2-2s}(y^{-1} \circ x) dy 
\\
& = - \frac{1}{\G(1-s)} \int_0^\infty t^{-s} \int_{\mathbb G} \LL u(y) p(y^{-1} \circ x,t) dy dt
\\
& = - \frac{1}{\G(1-s)} \int_0^\infty t^{-s} \LL_x \int_{\mathbb G} u(y) p(y^{-1} \circ x,t) dy dt
\\
& = - \frac{1}{\G(1-s)} \int_0^\infty t^{-s} \LL_x P_t u(x) dt
\\
& = - \frac{1}{\G(1-s)} \int_0^\infty t^{-s} \frac{d}{dt} [P_t u(x) - u(x)] dt
\\
& = - \frac{s}{\G(1-s)} \int_0^\infty t^{-1 -s} [P_t u(x) - u(x)] dt.
\end{align*}
We notice that the integration by parts in the last equality is justified by the fact that
\[
\underset{t\to 0^+}{\lim} \left\| \frac{P_t u - u}{t} - \LL u\right\|_{L^\infty(\mathbb G)} = 0,
\]
see (ii) in Theorem 3.1 in \cite{F}. Therefore, for every $0<b<1$ we have
\begin{equation}\label{decayG}
||P_t u - u||_{L^\infty(\mathbb G)} = o(t^b),\ \ \ \ \ \ \ \ \ \text{as}\ t\to 0^+.
\end{equation}
In particular, given $0<s<1$ and a point $x\in \mathbb G$, if we fix $b\in (s,1)$, then we have as $t\to 0^+$
\[
t^{-s} |P_t u(x) - u(x)|\le t^{-s} ||P_t u(x) - u(x)||_{L^\infty(\mathbb G)} \le C t^{b-s} \ \longrightarrow\ 0.
\]
Since on the other hand 
\[
t^{-s} |P_t u(x) - u(x)|\ t^{-s} ||P_t u(x) - u(x)||_{L^\infty(\mathbb G)}  \le 2||u||_{L^\infty(\mathbb G)} t^{-s}\ \longrightarrow\ 0
\]
as $t\to \infty$, we conclude that the above integration by parts is justified.   

\end{proof}

We mention that, in the special case of the Heisenberg group $\Hn$, there exists a different definition of fractional sub-Laplacean which seems better adapted to the sub-Riemannian geometry of $\Hn$. This is the conformal fractional sub-Laplacean introduced in the paper \cite{FGMT}. This latter operator arises as the Dirichlet-to-Neumann map of an \emph{extension operator} different from the one introduced in \cite{FF}, which is given by
\begin{equation}\label{extFF}
\LL_a = z^a(\LL + \Ba),
\end{equation} 
where $\Ba = \frac{\p^2}{\p z^2} + \frac az \frac{\p}{\p z}$ is the Bessel operator on the half-line $\{z>0\}$. 


\section{The parabolic extension for the fractional heat operator}\label{S:parext}

In order to understand some fundamental properties of the extension operator in the general setting of this note, we now take a detour into a parabolic version of \eqref{extFF}.
We begin by considering the Cauchy problem for the Bessel operator $\Ba$,  
 with Neumann boundary condition,  
\begin{equation}\label{CP}
\begin{cases}
\p_t u  - \Ba u = 0,\ \ \ \ \ \ \ \text{in} \ (0,\infty)\times (0,\infty),
\\
u(z,0) = \vf(z),\ \ \ \ \ \ \ \ z\in (0,\infty),
\\
\underset{z\to 0^+}{\lim} \paa u(z,t) = 0.
\end{cases}
\end{equation}
One has the following result, see e.g. Proposition 22.3 in \cite{GFT}. 

\begin{prop}\label{P:CP}
The solution of the Cauchy problem \eqref{CP} admits the representation formula 
\begin{equation}\label{repfor}
u(z,t) = P^{(a)}_t \vf(z) \overset{def}{=} \int_0^\infty \vf(\zeta) p^{(a)}(z,\zeta,t) \zeta^a d\zeta,
\end{equation}
where for $z,\zeta,t>0$ we have denoted by
\begin{align}\label{fs}
p^{(a)}(z,\zeta,t) & =(2t)^{-\frac{a+1}{2}}\left(\frac{z\zeta}{2t}\right)^{\frac{1-a}{2}}I_{\frac{a-1}{2}}\left(\frac{z\zeta}{2t}\right)e^{-\frac{z^2+\zeta^2}{4t}}
\\
& = \frac{1}{2t} (z\zeta)^{\frac{1-a}{2}}I_{\frac{a-1}{2}}\left(\frac{z\zeta}{2t}\right)e^{-\frac{z^2+\zeta^2}{4t}},
\notag
\end{align}
the heat kernel of $\Ba$ on $(\R^+,z^a dz)$, with Neumann boundary conditions.
\end{prop}

In \eqref{fs} we have denoted by $I_\nu(z)$ the modified Bessel function of the first kind and order $\nu\in \mathbb C$. The following two propositions can be found in \cite{Gaa}.

\begin{prop}[Stochastic completeness]\label{P:sc}
Let $a>-1$. For every $z\in \R^+$ and $t>0$ one has
\[
\int_0^\infty p^{(a)}(z,\zeta,t) \zeta^a d\zeta = 1.
\]
\end{prop}

\begin{prop}[Chapman-Kolmogorov equation]\label{P:semigroup}
Let $a>-1$. For every $z, \eta>0$ and every $0<s, t<\infty$ one has
\[
p^{(a)}(z,\eta,t) = \int_0^\infty p^{(a)}(z,\zeta,t) p^{(a)}(\zeta,\eta,s) \zeta^a d\zeta.
\]
\end{prop}

Propositions \ref{P:sc} and \ref{P:semigroup} prove that $\{P^{(a)}_t\}_{t>0}$ defines a Markovian semigroup of operators on $(0,\infty)$ with respect to the measure $d\mu = \zeta^a d\zeta$.

We next introduce the following local (doubly degenerate) operator which  constitutes the \emph{extension operator} for the fractional powers  $\HH^s$, $0<s<1$, where $\HH$ is given by \eqref{HO} above:
\begin{equation}\label{extFrHeat}
\HHa =  z^a(\HH + \Ba) = z^a \left(\frac{\partial}{\partial t} - \mathscr L + \Ba\right).
\end{equation}

In the classical setting when $\LL = \Delta$ the operator \eqref{extFrHeat} has been recently introduced in \cite{NS} and independently in \cite{ST}. In this same setting, the regularity theory has been extensively developed in \cite{BG} in connection with the study of the unique continuation problem. We mention that $\HHa$ belongs to a class of degenerate parabolic equations which was first introduced and studied by Chiarenza and Serapioni in \cite{CSe}.

 From the form of \eqref{extFrHeat}, and following the \emph{ansatz} in \cite{GL1}, we claim that the Neumann fundamental solution for $\HHa$, with singularity at a point $(Y,\tau) = (y,\zeta,\tau)\in \Rnp\times \R$, is given by
\begin{equation}\label{sGa}
\mathscr G_a(X,t;Y,\tau) = p(x,y,t-\tau) p^{(a)}(z,\zeta,t-\tau).
\end{equation}
We leave the verification of the claim to the interested reader. From Remark 22.4 in \cite{GFT}, we see that, in the special case when $Y = (y,0,\tau)$, i.e., $Y$ belongs to the thin manifold $\{z = 0\}$ on the boundary of $\Rnp\times (0,\infty)$, we have 
\begin{equation}\label{scaseGa}
\mathscr G_a((x,z,t);(y,0,\tau))= \frac{1}{2^a \G(\frac{a+1}{2})} (t-\tau)^{-\frac{a+1}{2}} e^{-\frac{z^2}{4(t-\tau)}} p(x,y,t-\tau). 
\end{equation}
If we consider the fundamental solution of the adjoint operator 
\[
\mathscr G_{-a}((x,z,t);(y,0,0))= \frac{1}{2^{-a} \G(\frac{1-a}{2})} t^{-\frac{1-a}{2}} e^{-\frac{z^2}{4t}}  p(x,y,t),
\]
then we easily recognize that
\begin{align}\label{poissonGa}
 - z^{-a} \p_z \mathscr G_{-a}((x,z,t);(y,0,0)) & = \frac{1}{2^{1-a} \G(\frac{1-a}{2})} \frac{z^{1-a}}{t^{\frac{1-a}{2}+1}} e^{-\frac{z^2}{4t}}  p(x,y,t)
\\
& = \frac{1}{2^{1-a} \G(\frac{1-a}{2})} \frac{z^{1-a}}{t^{\frac{3-a}{2}}} e^{-\frac{z^2}{4t}}  p(x,y,t).
\notag
 \end{align}
 
\begin{dfn}\label{D:extheatpoisson}
We define the \emph{Poisson kernel} for the operator $\HHa$ in \eqref{extFrHeat} above as the function
\begin{equation}\label{heatPK}
P^{(a)}_z(x,y,t) = \frac{1}{2^{1-a} \G(\frac{1-a}{2})} \frac{z^{1-a}}{t^{\frac{3-a}{2}}} e^{-\frac{z^2}{4t}}  p(x,y,t).
\end{equation}
\end{dfn}

We mention that in the classical case when $\LL = \Delta$, the standard Laplacean, and therefore $p(x,y,t) = (4\pi t)^{-\frac n2} \exp\left(-\frac{|x-y|^2}{4t}\right)$, the formula \eqref{heatPK} first appeared on p. 309 of the paper \cite{AC}. A first basic property of the kernel $P^{(a)}_z(x,y,t)$, which is a consequence of the basic property \eqref{P1} above, is given by the following proposition.

\begin{prop}\label{P:heatPKone}
For every $(x,z)\in \Rnp$ one has
\[
\int_0^\infty \int_{\Rn} P^{(a)}_z(x,y,t) dy dt = 1.
\]
\end{prop}

\begin{proof}
Using the stochastic completeness in \eqref{P1}, we obtain
\begin{align*}
\int_0^\infty \int_{\Rn} P^{(a)}_z(x,y,t) dy dt & =  \frac{1}{2^{1-a} \G(\frac{1-a}{2})} \int_0^\infty \frac{z^{1-a}}{t^{\frac{3-a}{2}}} e^{-\frac{z^2}{4t}} \left(\int_{\Rn} p(x,y,t) dy\right) dt
\\
& = \frac{1}{2^{1-a} \G(\frac{1-a}{2})} \int_0^\infty \frac{z^{1-a}}{t^{\frac{3-a}{2}}} e^{-\frac{z^2}{4t}} dt.
\end{align*}
Recalling that $-1<a<1$, we easily see that the integral in the right-hand side of the latter equation is convergent and an easy calculation gives 
\[
\int_0^\infty \frac{z^{1-a}}{t^{\frac{3-a}{2}}} e^{-\frac{z^2}{4t}} dt = 2^{1-a} \G\left(\frac{1-a}{2}\right).
\]

\end{proof}

We next address the question: what equation does $P^{(a)}_z(x,y,t)$ satisfy?

\begin{prop}\label{P:eqheatPK}
For every $x, y\in \Rn$, $x\not= y$, and $t>0$ one has
\begin{equation}\label{heatPKfinal}
\p_t P^{(a)}_z(x,y,t) - \Ba P^{(a)}_z(x,y,t) = \LL_x P^{(a)}_z(x,y,t).
\end{equation}
\end{prop}
 
\begin{proof}
We obtain from the definition \eqref{heatPK}  
\begin{equation}\label{eqheatPK}
\p_z P^{(a)}_z(x,y,t) = \frac{1-a}z P^{(a)}_z(x,y,t) -\frac{z}{2t} P^{(a)}_z(x,y,t).
\end{equation}
This gives
\begin{align}\label{eqheatPK2}
\p_{zz} P^{(a)}_z(x,y,t) & = \frac{1-a}z \p_z P^{(a)}_z(x,y,t) - \frac{1-a}{z^2} P^{(a)}_z(x,y,t) 
\\
& -\frac{1}{2t} P^{(a)}_z(x,y,t) -\frac{z}{2t} \p_z P^{(a)}_z(x,y,t).
\notag
\end{align}
Substituting \eqref{eqheatPK} into \eqref{eqheatPK2}, we find
\begin{align}\label{heatPKdzz}
\p_{zz} P^{(a)}_z(x,y,t) & = \frac{1-a}z\left(\frac{1-a}z P^{(a)}_z(x,y,t) -\frac{z}{2t} P^{(a)}_z(x,y,t)\right)
\\
& - \frac{1-a}{z^2} P^{(a)}_z(x,y,t) -\frac{1}{2t} P^{(a)}_z(x,y,t)
\notag\\
& -\frac{z}{2t} \left(\frac{1-a}z P^{(a)}_z(x,y,t) -\frac{z}{2t} P^{(a)}_z(x,y,t)\right)
\notag\\
& = \frac{(1-a)^2}{z^2} P^{(a)}_z(x,y,t) -\frac{1-a}{t} P^{(a)}_z(x,y,t)
\notag\\
& - \frac{1-a}{z^2} P^{(a)}_z(x,y,t) -\frac{1}{2t} P^{(a)}_z(x,y,t) + \frac{z^2}{4t^2} P^{(a)}_z(x,y,t).
\notag
\end{align}
Combining \eqref{heatPKdzz} and \eqref{eqheatPK2} we  obtain
\begin{equation}\label{heatPKdzz2}
\Ba P^{(a)}_z(x,y,t) = \left(\frac{z^2}{4t^2} - \frac{3-a}{2t}\right) P^{(a)}_z(x,y,t).
\end{equation} 
Next, differentiating \eqref{heatPK} with respect to $t$, and using the equation $\p_t p = \LL_x p $ satisfied by the fundamental solution $p(x,y,t)$, see  \eqref{Hp} above, we find
\begin{equation}\label{heatPKdt}
\p_t P^{(a)}_z(x,y,t) = \left(\frac{z^2}{4t^2} - \frac{3-a}{2t}\right) P^{(a)}_z(x,y,t) + \LL_x P^{(a)}_z(x,y,t).
\end{equation} 
The equations \eqref{heatPKdzz2} and \eqref{heatPKdt} finally give \eqref{heatPKfinal}.

\end{proof}


\section{Solution of the extension problem for $(-\LL)^s$}\label{S:extpb}

In this final section we use the parabolic extension Poisson kernel $P^{(a)}_z(x,y,t) $ in \eqref{heatPK}
above to introduce the Poisson kernel for the subelliptic extension operator $\LL_a$.

\begin{dfn}\label{D:extpoisson}
The \emph{Poisson kernel} for the operator $\LL_a$ in \eqref{extFF} above is defined as
\begin{equation}\label{slPK}
K^{(a)}_z(x,y) = \int_0^\infty P^{(a)}_z(x,y,t) dt = \frac{1}{2^{1-a} \G(\frac{1-a}{2})} z^{1-a}\int_0^\infty \frac{e^{-\frac{z^2}{4t}} }{t^{\frac{3-a}{2}}}  p(x,y,t) dt.
\end{equation}
\end{dfn} 

\begin{rmrk}\label{R:PK}
We emphasize that, when $\LL = \Delta$, formula \eqref{slPK} gives back the Caffarelli-Silvestre Poisson kernel 
\[
P_s(x,y) = \frac{\G(\frac n2 + s)}{\pi^{\frac n2}\G(s)} \frac{y^{2s}}{(y^2 + |x|^2)^{\frac{n+2s}{2}}}
\]
for the extension operator for $(-\Delta)^s$ in \cite{CS}. We also mention that, although as we have previously mentioned in the special setting of Carnot groups in \cite{FF} the authors work with the definition \eqref{FF}, which seemingly differs from our \eqref{flheat}, interestingly in (26) of their Theorem 4.4 they obtain precisely the same Poisson kernel as in \eqref{slPK} above.
\end{rmrk}

The following basic property of the kernel $K^{(a)}_z(x,y)$ is an immediate consequence of the definition \eqref{slPK} and of Proposition \ref{P:heatPKone}.

\begin{prop}\label{P:PKone}
For every $(x,z)\in \Rnp$ one has
\[
\int_{\Rn} K^{(a)}_z(x,y) dy = 1.
\]
\end{prop}

We next prove that the kernel  $K^{(a)}_z(x,y)$ is a solution of the extension operator $\LL_a$ in \eqref{extFF} above. 

\begin{prop}\label{P:PKextsolution} 
Fix $y\in \Rn$. For every $x\not= y$ and $z>0$ one has
\[
\LL_{a,x} K^{(a)}_z(x,y) = 0.
\]
\end{prop}

\begin{proof}
Using \eqref{extFF} we find for any $z>0$ and $x\not = y$
\[
z^{-a} \LL_{a,x} K^{(a)}_z(x,y) = \LL_x K^{(a)}_z(x,y) + \Ba K^{(a)}_z(x,y).
\]
To compute the quantities in the right-hand side of the latter equation we next differentiate with respect to $x$ under the integral sign in \eqref{slPK}. Such operation can be justified using the definition \eqref{heatPK} of $P^{(a)}_z(x,y,t)$ and the Gaussian estimates in Theorem \ref{T:KS}. We obtain 
\begin{equation}\label{KaLL}
z^{-a} \LL_{a,x} K^{(a)}_z(x,y) = \int_0^\infty \LL_{x}  P^{(a)}_z(x,y,t) dt + \int_0^\infty \Ba  P^{(a)}_z(x,y,t) dt.
\end{equation}
To compute the first integral in the right-hand side of the latter equation we now use \eqref{heatPKfinal} in Proposition \ref{P:eqheatPK} which gives for every $x, y\in \Rn$, $x\not=y$, and $t>0$,
\begin{align}\label{KaLL2}
\int_0^\infty \LL_{x}  P^{(a)}_z(x,y,t) dt & = \int_0^\infty \p_t P^{(a)}_z(x,y,t) dt -  \int_0^\infty \Ba P^{(a)}_z(x,y,t) dt
\\
& = - \int_0^\infty \Ba P^{(a)}_z(x,y,t) dt,
\notag
\end{align}
since by \eqref{heatPK} and Theorem \ref{T:KS} we have for every $x\not= y$
\[
\int_0^\infty \p_t P^{(a)}_z(x,y,t) dt = 0.
\]
Substituting \eqref{KaLL2} in \eqref{KaLL} we reach the desired conclusion. 

\end{proof}

\begin{dfn}[The extension problem]\label{D:ep}
The extension problem in $\Rnp$ for the nonlocal operator $(-\LL)^s$, $0<s<1$, is the following:
\begin{equation}\label{ep}
\begin{cases}
\LL_a U = 0,
\\
U(x,0) = u(x).
\end{cases}
\end{equation}
\end{dfn}

We next show how to solve \eqref{ep}. 
Given $u\in \mathscr S(\Rn)$ we define
\begin{equation}\label{pos}
U(x,z) = \int_{\R^{n}} K^{(a)}_{z}(x,y) u(y) dy.
\end{equation}

\begin{prop}\label{P:UsolvesextL2}
The function $U$ defined by \eqref{pos} solves the extension problem \eqref{ep}, in the sense that $\LL_a U = 0$ in $\Rnp$, and we have in $L^2(\Rn)$
\begin{equation}\label{UL2}
\underset{z\to 0^+}{\lim} U(\cdot;z) = u.
\end{equation}
\end{prop}

\begin{proof}
Differentiating under the integral sign and using Proposition \ref{P:PKextsolution} it is clear that $U$ solves the equation $\LL_a U = 0$ in $\Rnp$. To prove \eqref{UL2} we argue as follows. In view of Proposition \ref{P:PKone} we have for every $x\in \Rn$
\[
U(x,z) - u(x) = \int_{\R^{n}} K^{(a)}_{z}(x,y) [u(y) - u(x)] dy = \int_0^\infty \int_{\R^{n}} P^{(a)}_{z}(x,y,t) [u(y) - u(x)] dy dt,
\]
where in the second equality we have used \eqref{slPK}. By the definition of $P^{(a)}_{z}(x,y,t)$ we further obtain
\begin{align}\label{Uminusu}
U(x,z) - u(x) & = \frac{1}{2^{1-a} \G(\frac{1-a}{2})} z^{1-a} \int_0^\infty \frac{e^{-\frac{z^2}{4t}} }{t^{\frac{3-a}{2}}}  \int_{\R^{n}} p(x,y,t) [u(y) - u(x)] dy dt,
\\
& = \frac{1}{2^{1-a} \G(\frac{1-a}{2})} z^{1-a} \int_0^\infty \frac{e^{-\frac{z^2}{4t}} }{t^{\frac{3-a}{2}}}  [P_t u(x) - u(x)] dt.
\notag
\end{align}
In what follows, in order to simplify the notation we indicate with $||\ ||$ the norm of a function in $L^2(\Rn)$. Formula \eqref{Uminusu} gives
\begin{align*}
||U(\cdot;z) - u|| & \le \frac{1}{2^{1-a} \G(\frac{1-a}{2})} z^{1-a} \int_0^\infty \frac{e^{-\frac{z^2}{4t}} }{t^{\frac{3-a}{2}}}  ||P_t u - u|| dt.
\end{align*}
We infer that \eqref{UL2} will be proved if we show that the right-hand side in the latter inequality tends to $0$ as $z\to 0^+$. With this objective in mind we write
\[
z^{1-a}  \int_0^\infty \frac{e^{-\frac{z^2}{4t}} }{t^{\frac{3-a}{2}}}  ||P_t u - u|| dt = z^{1-a}  \int_0^1 \frac{e^{-\frac{z^2}{4t}} }{t^{\frac{3-a}{2}}}  ||P_t u - u|| dt + z^{1-a}  \int_1^\infty \frac{e^{-\frac{z^2}{4t}} }{t^{\frac{3-a}{2}}}  ||P_t u - u|| dt.
\]
Since $||P_t u - u||\le ||P_t u|| + ||u||\le 2 ||u||$, and $\frac{3-a}{2}>1$, it is clear that 
\[
z^{1-a}  \int_1^\infty \frac{e^{-\frac{z^2}{4t}} }{t^{\frac{3-a}{2}}}  ||P_t u - u|| dt \le 2||u|| \int_1^\infty \frac{dt}{t^{\frac{3-a}{2}}} \le C(u,a) z^{1-a}\ \longrightarrow\ 0
\]
as $z\to 0^+$ since $1-a>0$. 
Next, we choose $0<b<1$ such that $0<b<\frac{1-a}2$. Using \eqref{ig2} we can write
\begin{equation}\label{near0}
z^{1-a}  \int_0^1 \frac{e^{-\frac{z^2}{4t}} }{t^{\frac{3-a}{2}}}  ||P_t u - u|| dt \le C z^{1-a}  \int_0^1 \frac{e^{-\frac{z^2}{4t}} }{t^{\frac{3-a}{2}-1-b}} \frac{dt}t.
\end{equation}
We now make the change of variable $\sigma =   \frac{z^2}{4t}$, for which $\frac{d\sigma}\sigma = - \frac{dt}t$, obtaining 
\begin{align*}
& z^{1-a}  \int_0^1 \frac{e^{-\frac{z^2}{4t}} }{t^{\frac{3-a}{2}-1-b}} \frac{dt}t = z^{1-a}  \int_{\frac{z^2}4}^\infty \left(\frac{z^2}{4\sigma}\right)^{-\frac{3-a}{2} +1 + b} e^{-\sigma}  \frac{d\sigma}\sigma
\\
& = C(a,b) z^{2b} \int_{\frac{z^2}4}^\infty \left(\frac{1}{\sigma}\right)^{-\frac{3-a}{2} +2 + b} e^{-\sigma}  d\sigma \le C(a,b) z^{2b} \int_{0}^\infty \left(\frac{1}{\sigma}\right)^{-\frac{3-a}{2} +2 + b} e^{-\sigma}  d\sigma\ \longrightarrow\ 0,
\end{align*}
as $z\to 0^+$, since the integral in the right-hand side converges if $-\frac{3-a}{2} +2 + b<1$, or equivalently $b< \frac{1-a}2$, which is true by our choice of $b$.

\end{proof}

Our next result shows that the Dirichlet datum $u$ is not just attained in $L^2(\Rn)$, but in the classical pointwise sense. 

\begin{prop}\label{P:Usolvesext}
The function $U$ defined by \eqref{pos} solves the extension problem \eqref{ep} in the sense that for every $x_0\in \Rn$ one has
\begin{equation}\label{Ulimitupoint}
\underset{(x,z)\to (x_0,0)}{\lim} U(x,z) = u(x_0).
\end{equation}
\end{prop}

\begin{proof}
To see that $U(x,z)$ satisfies \eqref{Ulimitupoint} we plan to show that for every $\ve>0$ there exists $\delta = \delta(x_0,\ve)>0$ such that
\begin{equation}\label{lim2}
d(x,x_0) < \delta,\ \ \ \ 0<z<\delta^2\ \Longrightarrow\ \left|U(x,z) - u(x_0)\right| < \ve.
\end{equation}
Now, given $x_0\in \Rn$ and $\ve>0$, we choose $\delta = \delta(x_0,\ve)>0$ such that $d(y,x_0)<\delta \Longrightarrow \left|u(y) - u(x_0)\right| < \frac{\ve}2$ (in fact, one should notice that $\delta$ can be taken independent of $x_0$). In view of Proposition \ref{P:heatPKone} this gives
\[
\int_0^\infty \int_{d(y,x_0)<\delta} P^{(a)}_{z}(x,y,t) \left|u(y) - u(x_0)\right| dy dt < \frac{\ve}2 \int_0^\infty \int_{\Rn} P^{(a)}_{z}(x,y,t) dy dt = \frac{\ve}2.
\]
Applying Proposition \ref{P:PKone}, we have 
\begin{align*}
\left|U(x,z) - u(x_0)\right| & \le \int_{\Rn} K^{(a)}_{z}(x,y) \left|u(y) - u(x_0)\right| dy
\\
& = \int_0^\infty \int_{\Rn} P^{(a)}_{z}(x,y,t) \left|u(y) - u(x_0)\right| dy dt
\\
& = \int_0^\infty \int_{d(y,x_0)<\delta} P^{(a)}_{z}(x,y,t) \left|u(y) - u(x_0)\right| dy dt
\\
& + \int_0^\infty \int_{d(y,x_0)\ge \delta} P^{(a)}_{z}(x,y,t) \left|u(y) - u(x_0)\right| dy dt
\\
& < \frac{\ve}2 + \int_0^\infty \int_{d(y,x_0)\ge \delta} P^{(a)}_{z}(x,y,t) \left|u(y) - u(x_0)\right| dy dt.
\end{align*}

On the other hand, we trivially have
\begin{align*}
& \int_0^\infty \int_{d(y,x_0) \ge \delta} P^{(a)}_{z}(x,y,t) \left|u(y) - u(x_0)\right| dy dt \le 2 ||u||_{L^\infty(\Rn)} \int_0^\infty \int_{d(y,x_0) \ge \delta} P^{(a)}_{z}(x,y,t) dy dt
\\
& = 2 ||u||_{L^\infty(\Rn)} \frac{z^{1-a}}{2^{1-a} \G(\frac{1-a}{2})} \int_0^\infty \frac{1}{t^{\frac{3-a}{2}}} e^{-\frac{z^2}{4t}} \int_{d(y,x_0) \ge \delta} p(x,y,t) dy dt.
\end{align*}
Now suppose that $d(x,x_0)<\frac{\delta}2$. Then, on the set where $d(y,x_0) \ge \delta$ we have
\[
d(y,x_0) \le d(y,x) + d(x,x_0) < d(y,x) + \frac{\delta}2 \le d(y,x) + \frac{d(y,x_0)}2.
\]
Therefore, on such set we have $\frac{d(y,x_0)}2 < d(y,x)$. This implies that, when $d(x,x_0)<\frac{\delta}2$, then
\[
\{y\in \Rn\mid d(y,x_0)\ge \delta\}\ \subset\ \{y\in \Rn\mid d(y,x) \ge \frac{\delta}2\}.
\]
Using now the upper Gaussian estimate in \eqref{gaussian1}, we have on the set $\{y\in \Rn\mid d(y,x) \ge \frac{\delta}2\}$,
\begin{align*}
p(x,y,t) & \le \frac{C}{|B(x,\sqrt{t})|} \exp \left(- \frac{M d(x,y)^2}{t}\right) 
\\
& = \frac{C}{|B(x,\sqrt{t})|} \exp \left(- \frac{M d(y,x)^2}{2t}\right) \exp \left(- \frac{M d(y,x)^2}{2t}\right)  
\\
& \le \exp \left(- \frac{M \delta^2}{8t}\right) \frac{C}{|B(x,\sqrt{t})|} \exp \left(- \frac{M d(y,x)^2}{2t}\right)
\end{align*}
This gives
\begin{align*}
& \int_{d(y,x_0) \ge \delta} p(x,y,t) dy  \le \int_{d(y,x) \ge \frac{\delta}2} p(x,y,t) dy 
\\
& \le \frac{C}{|B(x,\sqrt{t})|} \exp \left(- \frac{M \delta^2}{8t}\right) \int_{d(y,x) \ge \frac{\delta}2} \exp \left(- \frac{M d(y,x)^2}{2t}\right) dy
\\
& \le \frac{C}{|B(x,\sqrt{t})|} \exp \left(- \frac{M \delta^2}{8t}\right) \int_{\Rn} \exp \left(- \frac{M d(y,x)^2}{2t}\right) dy. 
\end{align*}
We now have from \eqref{exp}, for some constant $C^\star = C^\star(C_d,M)>0$,
\begin{equation*}
\int_{\Rn} \exp\left(-\frac{M d(y,x)^2}{2t}\right) dy \le C^\star |B(x,\sqrt t)|.
\end{equation*}
We conclude for some $\overline C>0$
\[
\int_{d(y,x_0) \ge \delta} p(x,y,t) dy  \le \overline C  \exp \left(- \frac{M \delta^2}{8t}\right).
\]
We thus find
\[
\int_0^\infty \frac{z^{1-a}}{t^{\frac{3-a}{2}}} e^{-\frac{z^2}{4t}} \int_{d(y,x_0) \ge \delta} p(x,y,t) dy dt \le \overline C \int_0^\infty \frac{z^{1-a}}{t^{\frac{3-a}{2}}} e^{-\frac{z^2}{4t}} \exp \left(- \frac{M \delta^2}{8t}\right) dt.
\]
To estimate the integral in the right-hand side of the latter inequality we make the change of variable $\sigma = \frac{4t}{z^2}$, obtaining
\begin{align*}
&  \int_0^\infty \frac{z^{1-a}}{t^{\frac{3-a}{2}}} e^{-\frac{z^2}{4t}} \exp \left(- \frac{M \delta^2}{8t}\right) dt = 2^{1-a} \int_0^\infty \frac{1}{\sigma^{\frac{3-a}{2}}} e^{-\frac{1}{\sigma}} \exp \left(- \frac{M \delta^2}{2z^2 \sigma}\right) d\sigma 
\\
& = 2^{1-a} \int_0^{\frac{1}{\delta}} \frac{1}{\sigma^{\frac{3-a}{2}}} e^{-\frac{1}{\sigma}} \exp \left(- \frac{M \delta^2}{2z^2 \sigma}\right) d\sigma + 2^{1-a} \int_{\frac{1}{\delta}}^\infty \frac{1}{\sigma^{\frac{3-a}{2}}} e^{-\frac{1}{\sigma}} \exp \left(- \frac{M \delta^2}{2z^2 \sigma}\right) d\sigma
\\
& = I_z(\delta) + II_z(\delta). 
\end{align*}
Suppose now that $0<z<\delta^2$. On the set where $0<\sigma<\frac{1}{\delta}$ we have $\frac{M \delta^2}{2z^2 \sigma} > \frac{M}{2\delta}$, and thus
\begin{align*}
I_z(\delta) & \le 2^{1-a} \int_0^{\frac{1}{\delta}} \frac{1}{\sigma^{\frac{3-a}{2}}} e^{-\frac{1}{\sigma}} \exp \left(- \frac{M \delta^2}{2z^2 \sigma}\right) d\sigma
\\
& \le \exp \left(- \frac{M}{2\delta}\right) 2^{1-a} \int_0^{\frac{1}{\delta}} \frac{1}{\sigma^{\frac{3-a}{2}}} e^{-\frac{1}{\sigma}}  d\sigma \le \exp \left(- \frac{M}{2\delta}\right) 2^{1-a} \int_0^{\infty} \frac{1}{\sigma^{\frac{1-a}{2}}} e^{-\frac{1}{\sigma}}  \frac{d\sigma}{\sigma}
\\
& = \exp \left(- \frac{M}{2\delta}\right) 2^{1-a} \int_0^{\infty} w^{\frac{1-a}{2}} e^{-w}  \frac{dw}{w} = 2^{1-a} \Gamma\left(\frac{1-a}{2}\right) \exp \left(- \frac{M}{2\delta}\right)\ \longrightarrow\ 0,
\end{align*}
as $\delta\to 0^+$. On the other hand, on the set where $\frac{1}{\delta}<\sigma<\infty$ we simply estimate $\exp \left(- \frac{M \delta^2}{2z^2 \sigma}\right) \le 1$, obtaining 
\begin{align*}
II_z(\delta) & \le 2^{1-a} \int_{\frac{1}{\delta}}^\infty \frac{1}{\sigma^{\frac{3-a}{2}}} e^{-\frac{1}{\sigma}} d\sigma
 = 2^{1-a} \int_{\frac{1}{\delta}}^\infty  \frac{1}{\sigma^{\frac{1-a}{2}}} e^{-\frac{1}{\sigma}}  \frac{d\sigma}{\sigma}
 \\
 & = 2^{1-a} \int_{\frac{1}{\delta}}^\infty  \frac{1}{\sigma^{\frac{1-a}{2}}} e^{-\frac{1}{\sigma}}  \frac{d\sigma}{\sigma} = 2^{1-a} \int_{0}^\delta  w^{\frac{1-a}{2}} e^{-w}  \frac{dw}{w} \ \longrightarrow\ 0,
\end{align*}
as $\delta\to 0^+$.
Therefore, given $\ve>0$ it suffices to further restrict $\delta>0$ in order to achieve \eqref{lim2}, for $d(x,x_0)<\delta$ and $0<z<\delta^2$.

\end{proof}

We finally prove that the fractional powers $(-\LL)^s$ introduced in \eqref{flheat} of Definition \ref{D:flheat} above are obtained as the Dirichlet-to-Neumann map of the extension problem \eqref{ep} above.

\begin{prop}\label{P:trace} 
Let $0<s<1$ and $a = 1-2s$. Given a function $u\in \mathscr S(\Rn)$, with $(-\LL)^s u(x)$ defined as in \eqref{flheat} above, one has in $L^2(\Rn)$
\begin{equation}\label{trace}
- \frac{2^{-a} \Gamma\left(\frac{1-a}2\right)}{\Gamma\left(\frac{1+a}2\right)}  \underset{z\to 0^+}{\lim} \paa U(x,z) = (-\LL)^s u(x).
\end{equation}
\end{prop}

\begin{proof}
We begin by noting that proving \eqref{trace} is equivalent to establishing the following in $L^2(\Rn)$:
\begin{equation*}
- \frac{2^{-a} \Gamma\left(\frac{1-a}2\right)}{\Gamma\left(\frac{1+a}2\right)}  \underset{z\to 0^+}{\lim} \paa U(\cdot,z) = - \frac{1-a}{2\G\left(\frac{1+a}2\right)} \int_0^\infty t^{-\frac{3-a}2} \left[P_t u - u\right] dt.
\end{equation*}
In view of our hypothesis \eqref{P1}, this is in turn equivalent to the equation
\begin{equation}\label{trace2}
2^{1-a} \Gamma\left(\frac{1-a}2\right)  \underset{z\to 0^+}{\lim} \paa U(\cdot,z) =  (1-a) \int_0^\infty \int_{\Rn} t^{-\frac{3-a}2}  p(\cdot,y,t) [u(y) - u(\cdot)] dy dt.
\end{equation}
We are thus left with verifying \eqref{trace2} in $L^2(\Rn)$. In order to achieve this we observe that \eqref{pos} above and Proposition \ref{P:PKone} allow us to write
\[
U(x,z) = \int_{\R^{n}} K^{(a)}_{z}(x,y) [u(y) - u(x)] dy + u(x).
\]
Therefore, if we differentiate under the integral sign in this latter equation we find
\begin{align*}
& 2^{1-a} \Gamma\left(\frac{1-a}2\right) \paa U(x,z) = 2^{1-a} \Gamma\left(\frac{1-a}2\right) \int_{\R^{n}} z^a \p_z K^{(a)}_{z}(x,y) [u(y) - u(x)] dy
\\
& = 2^{1-a} \Gamma\left(\frac{1-a}2\right) \int_0^\infty \int_{\R^{n}} z^a \p_z P^{(a)}_{z}(x,y) [u(y) - u(x)] dy dt,
\end{align*}
where in the last equality we have applied the definition \eqref{slPK}
of the Poisson kernel $K^{(a)}_z(x,y)$. 
We now apply the equation \eqref{eqheatPK} above, that gives
\[
z^a \p_z P^{(a)}_z(x,y,t) = (1-a) z^{a-1} P^{(a)}_z(x,y,t) -\frac{z^{a+1}}{2t} P^{(a)}_z(x,y,t).
\]
Substituting the latter expression in the above equation, and using \eqref{heatPK}, we thus find
\begin{align*}
& 2^{1-a} \Gamma\left(\frac{1-a}2\right) \paa U(x,z) 
\\
& = (1-a)  \int_0^\infty \int_{\R^{n}} t^{-\frac{3-a}{2}} e^{-\frac{z^2}{4t}}  p(x,y,t) [u(y) - u(x)] dy dt
\\
& - \frac{z^2}{2} \int_0^\infty \int_{\R^{n}} t^{-\frac{3-a}{2} -1} e^{-\frac{z^2}{4t}}  p(x,y,t) [u(y) - u(x)] dy dt.
\end{align*}
The proof of \eqref{lim2} will be completed if we can show that  in $L^2(\Rn)$
\begin{equation}\label{lim3}
\underset{z\to 0^+}{\lim} \int_0^\infty t^{-\frac{3-a}{2}} e^{-\frac{z^2}{4t}}  \left[P_t u - u\right] dt = \int_0^\infty t^{-\frac{3-a}{2}} \left[P_t u - u\right] dt,
\end{equation}
and
\begin{equation}\label{lim4}
\underset{z\to 0^+}{\lim} \frac{z^2}{2} \int_0^\infty t^{-\frac{3-a}{2} -1} e^{-\frac{z^2}{4t}} \left[P_t u - u\right]  dt = 0.
\end{equation} 
We begin with \eqref{lim3}. In what follows, in order to simplify the notation we indicate with $||\ ||$ the norm of a function in $L^2(\Rn)$. We have
\begin{align*}
& \left\|\int_0^\infty t^{-\frac{3-a}{2}} e^{-\frac{z^2}{4t}}  \left[P_t u - u\right] dt  - \int_0^\infty t^{-\frac{3-a}{2}} \left[P_t u - u\right] dt\right\|
\\
& \le \int_0^\infty t^{-\frac{3-a}{2}} \left(e^{-\frac{z^2}{4t}} - 1\right) \left\|P_t u - u\right\| dt.
\end{align*}
Let now $z_k\searrow 0^+$ and consider the sequence of functions on $(0,\infty)$
\[
g_k(t) \overset{def}{=} t^{-\frac{3-a}{2}} \left(e^{-\frac{z_k^2}{4t}} - 1\right) \left\|P_t u - u\right\|.
\]
We clearly have $0\le g_k(t) \to 0$ as $k\to \infty$, for every $t\in (0,\infty)$. Furthermore, since 
\[
\left\|P_t u - u\right\|\le ||P_t u|| + ||u|| \le 2 ||u||,
\]
and since $\frac{3-a}{2}>1$, we have for every $k\in \mathbb N$,
\[
g_k(t) \le 2 ||u|| t^{-\frac{3-a}{2}} \in L^1(1,\infty).
\]
Since $0<\frac{1-a}{2}<1$, we now choose $b\in (\frac{1-a}{2},1)$. By \eqref{ig2} we have $||P_t u - u|| = O(t^b)$ on $(0,1)$. We thus infer that there exists a constant $C>0$, independent of $k$, such that
\[
g_k(t) \le C t^{-\frac{3-a}{2} + b} \in L^1(0,1).
\] 
Therefore, the functions $g_k$ have a common dominant in $L^1(0,\infty)$. By Lebesgue dominated convergence theorem we conclude that \eqref{lim3} does hold.

Finally, to prove \eqref{lim4} we argue in a similar way. 
We have
\begin{align}\label{lim5}
& \left\|z^2 \int_0^\infty t^{-\frac{3-a}{2} -1} e^{-\frac{z^2}{4t}} \left[P_t u - u\right]  dt\right\|  \le z^2 \int_0^\infty t^{-\frac{3-a}{2} -1} e^{-\frac{z^2}{4t}} \left\|P_t u - u\right\|  dt
\\
& = z^2 \int_0^1 t^{-\frac{3-a}{2} -1} e^{-\frac{z^2}{4t}} \left\|P_t u - u\right\|  dt + z^2 \int_1^\infty t^{-\frac{3-a}{2} -1} e^{-\frac{z^2}{4t}} \left\|P_t u - u\right\|  dt
\notag\\
& \le  C_1 z^2  \int_0^1 t^{-\frac{3-a}{2} -1 + b} e^{-\frac{z^2}{4t}}  dt + 2 ||u|| z^2 \int_1^\infty t^{-\frac{3-a}{2} -1}  dt
\notag\\
& \le  C_1 z^2 \int_0^1 t^{-\frac{3-a}{2} -1 + b} e^{-\frac{z^2}{4t}}  dt + C_2 z^2.
\notag
\end{align}
Here, as before, $b\in (\frac{1-a}{2},1)$. Now, the change of variable $\sigma =   \frac{z^2}{4t}$, for which $\frac{d\sigma}\sigma = - \frac{dt}t$, gives
\begin{align*}
& \int_0^1 t^{-\frac{3-a}{2}  + b} e^{-\frac{z^2}{4t}}  \frac{dt}t = \int_{\frac{z^2}4}^\infty \left(\frac{z^2}{4\sigma}\right)^{-\frac{3-a}{2}  + b} e^{-\sigma}  \frac{d\sigma}\sigma \le C z^{2b -3 + a} \int_0^\infty \left(\frac{1}{\sigma}\right)^{-\frac{3-a}{2}  + b + 1} e^{-\sigma} d\sigma, 
\end{align*}
and the latter integral is finite if  $ b <  \frac{3-a}{2}$ (because of the factor $e^{-\sigma}$, there is of course no problem at infinity). But this is true, since $b < 1< 1 +\frac{1-a}{2} = \frac{3-a}{2}$. Fortunately, we still have a factor $z^2$ in front of the first integral in the right-hand side of \eqref{lim5}, and thus for such term $z$ is raised to the power
\[
 2 + 2b -3 + a = 2b + a - 1>0,
 \]
 since $b> \frac{1-a}2$! We conclude that also \eqref{lim4} does hold, thus completing the proof.

\end{proof}

The convergence in \eqref{trace} of Proposition \ref{P:trace} is in the $L^2$ sense. One may naturally wonder about pointwise convergence. With this objective in mind we next establish a useful pointwise estimate.

\begin{prop}\label{P:sqrtdecay}
Let $u\in \mathscr S(\Rn)$. Then, for every $x\in \Rn$ we have
\begin{equation}\label{Ptsqrt}
|P_t u(x) - u(x)| \le  C ||\x u||_{L^\infty(\Rn)} \sqrt t.
\end{equation}
\end{prop}

\begin{proof}
To prove  \eqref{Ptsqrt} we argue as follows. We have
\begin{align*}
P_t u(x) - u(x) & = \int_0^t \frac{d}{d\tau} P_\tau u(x) ds = \int_0^t \LL P_\tau u(x) d\tau = \int_0^t \int_{\Rn} \LL p(x,y,\tau) u(y)dy d\tau 
\\
& = - \int_0^t \int_{\Rn} < \x p(x,y,\tau),\x u(y)> dy d\tau.
\end{align*}
Using \eqref{gaussian1} in the above identity, we find
\begin{align*}
|P_t u(x) - u(x)| & \le  \int_0^t \int_{\Rn} |\x p(x,y,\tau)||\x u(y)| dy d\tau
\\
& \le   C ||\x u||_{L^\infty(\Rn)} \int_0^t  \frac{1}{\sqrt{\tau} |B(x,\sqrt{\tau})|} \left(\int_{\Rn}  \exp \bigg( - \frac{M d(x,y)^2}{\tau}\bigg) dy\right)d\tau
\end{align*}
Using now \eqref{exp} we conclude that \eqref{Ptsqrt} does hold.

\end{proof}

Using Proposition \ref{P:sqrtdecay} we can now pass from the $L^2$ convergence in Proposition \ref{P:trace} to a uniform pointwise one, at least in the regime $0<s<1/2$.

\begin{cor}\label{C:trace} 
Let  $0<s<1/2$ and $a = 1-2s$. Given a function $u\in \mathscr S(\Rn)$, with $(-\LL)^s u(x)$ defined as in \eqref{flheat} above, one has for every $x\in \Rn$
\begin{equation}\label{trace2}
- \frac{2^{-a} \Gamma\left(\frac{1-a}2\right)}{\Gamma\left(\frac{1+a}2\right)}  \underset{z\to 0^+}{\lim} \paa U(x,z) = (-\LL)^s u(x).
\end{equation}
\end{cor}

We omit the proof of Corollary \ref{C:trace}. We only confine ourselves to observe that \eqref{Ptsqrt} now guarantees, for every $x\in \Rn$, the summability of the integrand in the right-hand side of \eqref{flheat} in the range $0<s<1/2$ (notice that there is no issue for $t$ large since for any fixed $x\in \Rn$ one has the trivial bound
\[
|P_t u(x) - u(x)| \le ||P_t u||_{L^\infty(\Rn)} +  ||u||_{L^\infty(\Rn)} \le 2 ||u||_{L^\infty(\Rn)},
\]
by the fact that $P_t$ is sub-Markovian. However, the integrability of $t\to t^{-s-1} \left[P_t u(x) - u(x)\right]$ near $t = 0$ is subtler.   
Although this is verified in a number of situations, the question of convergence in the regime $1/2\le s <1$ is a bit delicate, as one needs a stronger decay in $t$ than that in \eqref{Ptsqrt}. 

Let us provide the reader with some motivation. Suppose that $\LL = \Delta$, the standard Laplacean in $\Rn$. Then, elementary considerations show that
\begin{equation}\label{Ptu}
P_t u(x) - u(x) = \int_{\Rn} G(y,t)[u(x+y) + u(x-y) - 2u(x)] dy,
\end{equation}
where we have indicated with $G(y,t) = (4\pi t)^{-\frac n2} e^{-\frac{|y|^2}{4t}}$ the Gauss-Weierstrass kernel. Having the second difference $u(x+y) + u(x-y) - 2u(x)$ is quite important for improving on \eqref{Ptsqrt}. If $u\in \mathscr S(\Rn)$, applying Taylor's formula with initial point $x$, we obtain for every $y\in \Rn$
\[
|u(x+y) + u(x-y) - 2u(x)| \le C  ||\nabla^2 u||_{L^\infty(\Rn)} |y|^2,
\]
where $C>0$ is universal, and $\nabla^2 u$ indicates the Hessian matrix of $u$.
We now use this information in \eqref{Ptu} in the following way
\begin{align*}
|P_t u(x) - u(x)| & \le \int_{|y|<\sqrt t} G(y,t)|u(x+y) + u(x-y) - 2u(x)| dy
\\
& + \int_{|y|\ge \sqrt t} G(y,t)|u(x+y) + u(x-y) - 2u(x)| dy
\\
& \le C ||\nabla^2 u||_{L^\infty(\Rn)} \left\{t \int_{|y|<\sqrt t} G(y,t) dy + \int_{|y|\ge \sqrt t} |y|^2 G(y,t) dy\right\}. 
\end{align*}
It is now easy to recognize that
\[
\int_{|y|\ge \sqrt t} |y|^2 G(y,t) dy \le C(n) t.
\]
We conclude that we now have
\begin{equation}\label{Pt2}
|P_t u(x) - u(x)| \le C ||\nabla^2 u||_{L^\infty(\Rn)}\  t.
\end{equation}
The improved estimate \eqref{Pt2} does now guarantee the integrability of $t\to t^{-s-1} \left[P_t u(x) - u(x)\right]$ near $t = 0$, thus establishing the validity of Corollary \ref{C:trace} in the whole range $0<s<1$ for the standard Laplacean. 

An improved decay which suffices to deal with the regime $1/2\le s<1$  does hold also in the setting of Carnot groups.  This is a direct consequence of \eqref{decayG} above. Therefore, Corollary \ref{C:trace} also holds in any Carnot group in the whole range $0<s<1$. For the more general operators treated in this note we will address this point in a forthcoming work.

\end{document}